\documentclass[12pt]{amsart}

\usepackage{amssymb, stmaryrd}
\usepackage{hyperref}
\usepackage[margin=1.0 in]{geometry}
\usepackage{color,soul}
\usepackage{multicol}
\usepackage[table]{xcolor}
\usepackage{enumerate}
\usepackage{mdframed}
\usepackage[T1]{fontenc}
\usepackage{tikz}
\usepackage{MnSymbol}
\usetikzlibrary{positioning}
\usepackage{multicol}
\usepackage{caption}
\usepackage{subcaption}
\usepackage{float}

\makeatletter
\def\bbordermatrix#1{\begingroup \m@th
  \@tempdima 4.75\p@
  \setbox\z@\vbox{%
    \def\cr{\crcr\noalign{\kern2\p@\global\let\cr\endline}}%
    \ialign{$##$\hfil\kern2\p@\kern\@tempdima&\thinspace\hfil$##$\hfil
      &&\quad\hfil$##$\hfil\crcr
      \omit\strut\hfil\crcr\noalign{\kern-\baselineskip}%
      #1\crcr\omit\strut\cr}}%
  \setbox\tw@\vbox{\unvcopy\z@\global\setbox\@ne\lastbox}%
  \setbox\tw@\hbox{\unhbox\@ne\unskip\global\setbox\@ne\lastbox}%
  \setbox\tw@\hbox{$\kern\wd\@ne\kern-\@tempdima\left[\kern-\wd\@ne
    \global\setbox\@ne\vbox{\box\@ne\kern2\p@}%
    \vcenter{\kern-\ht\@ne\unvbox\z@\kern-\baselineskip}\,\right]$}%
  \null\;\vbox{\kern\ht\@ne\box\tw@}\endgroup}
\makeatother

\def\VR{\kern-\arraycolsep\strut\vrule &\kern-\arraycolsep}
\def\vr{\kern-\arraycolsep & \kern-\arraycolsep}

\theoremstyle{plain}
\newtheorem{theorem}[subsection]{Theorem}
\newtheorem{proposition}[subsection]{Proposition}
\newtheorem{lemma}[subsection]{Lemma}

\newtheorem{corollary}[subsection]{Corollary}

\theoremstyle{definition}
\newtheorem{definition}[subsection]{Definition}

\newtheorem{claim}[subsection]{Claim}
\newtheorem{formula}[subsection]{Formula}

\newtheorem{example}[subsubsection]{Example}

\newtheorem{remark}[subsection]{Remark}

\normalfont
\usepackage[T1]{fontenc}{}

\newcommand{\fpt}{\operatorname{fpt}}

\title{Lower Bounds on the F-pure Threshold\\ and Extremal Singularities\footnote{Math Subject Classifications: Primary 13A35. Secondary 14B05.}
}
\author{Zhibek Kadyrsizova, Jennifer Kenkel, Janet Page, Jyoti Singh, \\ Karen E. Smith, Adela Vraciu, and Emily E. Witt  }
\email{ kesmith@umich.edu}
\thanks{This project  began at the AWM-sponsored workshop "Women in Commutative Algebra" at the Banff International Research Station. Partial funding for participants was supplied by 
  NSF grant numbers  193439 and  NSF-HRD 150048. In addition, partial funding was provided  by SERB(DST) grant number ECR/2017/000963 (for Jyoti Singh), NSF grant number 1801697, 1952399, and 2101075 (for Karen Smith), and NSF CAREER grant 1945611 (for Emily Witt).}

\begin{document}

\maketitle


\begin{abstract}
We prove that  if $f$ is a reduced  homogenous polynomial of degree $d$, then its $F$-pure threshold at the unique homogeneous maximal ideal  is at least $\frac{1}{d-1}$. We show, furthermore,  that its $F$-pure threshold equals $\frac{1}{d-1}$ if and only if $f\in \mathfrak m^{[q]}$ and $d=q+1$, where $q$ is a power of $p$.  Up to linear changes of coordinates (over a fixed algebraically closed field), we  classify  such "extremal singularities,"  and show that there is at most one with isolated singularity. Finally, we indicate several  ways in which  the projective hypersurfaces defined by such forms are 
 "extremal," for example, in terms of the configurations of lines they can contain.

\end{abstract}

\section{Introduction}

\noindent
 
Fix an algebraically closed  field $k$ of positive characteristic $p$.  What is the most singular possible hypersurface singularity  over $k$? 

The {\it  multiplicity} is the first  crude  measurement  of singularity---roughly speaking higher multiplicity singularities are more singular. But we want to identify which singularities are the most singular, even taking  multiplicity into account.  Among multiplicity two singularities, for example, the cusp $y^2-x^3$  is more singular than the normal crossing  $xy$,  but also  the cusp is more singular  in some characteristics than others. For example, the derivative with respect to $x$ vanishes to all orders when $k$ has characteristic three, but only to order two in any other characteristic.

\medskip
The {\it  $F$-pure threshold}  is a more refined  numerical invariant for comparing singularities of hypersurfaces  in positive characteristic. Analogous to---but much more subtle than---the log canonical threshold for a complex hypersurface,   the  $F$-pure threshold  is equal to one at each smooth  point (or more generally,  at each {\it $F$-pure} point),  with  "worse singularities" having {\it smaller} $F$-pure thresholds.   Using the $F$-pure threshold,  can we  identify a class of singularities that is "maximally singular"? What are the properties of such "extremal singularities"?

This  paper  proves   an elegant  lower bound on the $F$-pure threshold
 of a reduced homogeneous form in terms of its degree---that is, in terms of the multiplicity at the origin. We show that our  lower bound   is sharp, and  characterize  precisely what  forms achieve it.  One can argue that 
such hypersurfaces---those with minimal $F$-pure threshold---are  "maximally singular", and indeed, we prove
they have  several interesting  extremal algebraic and geometric properties. We also prove a complete classification of these {\bf extremal singularities}  up to  linear change of coordinates.

The main theorem of the first half of the paper is the following sharp bound on the F-pure threshold and 
characterization of the forms achieving it:

 \begin{theorem}\label{main1}
 Fix any field $k$ of  positive characteristic $p$.
Let $f\in k[x_1, \dots, x_n]$ be a  homogeneous polynomial of degree  $d=\mathrm{deg}(f)\geq 2$, reduced over the algebraic closure of $k$.
Then 
\begin{equation}\label{Lowerbound2}
\fpt(f) \ge \frac{1}{d-1}.
\end{equation}
Furthermore, equality  holds in  (\ref{Lowerbound2})  if and only if  $d=q+1$,  where $q$ is a power of $p$ and $f\in \langle x_1^{q}, \dots, x_n^{q}\rangle.$
\end{theorem}

Theorem \ref{main1} is a combination of Theorem
\ref{main}, in which we prove the general bound in terms of degree and find the degrees that can achieve it, 
and Theorem \ref{justify}, in which we characterize the {\it forms}  achieving the bound.
The most difficult step is showing that the bound  can be sharp only for certain $d$, which uses delicate analytic estimates in Section 2.  While there has been much research into computing   the $F$-pure threshold in specific settings (see, for example,  \cite{BhattSingh}, \cite{Hernandez}, \cite{HNWZ}, \cite{hernandez-teixeira}), we are not aware of any prior  research into { lower bounds} on the $F$-pure threshold.

The forms achieving the lower bound (\ref{Lowerbound2})  are { extremal} in many ways unrelated to F-pure threshold.  In  Section 8, we discuss several extremal geometric properties of the corresponding projective hypersurfaces.  For example,  when the  form is in at least four variables and defines a smooth hypersurface,
 these are the only projective hypersufaces (of degree larger than two) with the property that all smooth hyperplane sections are isomorphic; this follows from  Corollary \ref{cor1} and a theorem of Beauville (see Remark \ref{rem8}).  They  have  inseparable Gauss maps   and are isomorphic to their Gauss  duals; see Proposition \ref{gauss} for a precise statement and Remark \ref{just} for a discussion of how such behavoir is  extremal. 
Their extremal nature  is also reflected in the very special   configurations of  lines on them: despite containing many coplanar lines, no plane section  can contain a triangle (Proposition \ref{starconfig}).  This triangle-free property characterizes smooth cubic surfaces defined by forms of minimal F-pure threshold \cite{cubicspaper}.

There is no analog for such "extremely singular" singularities in characteristic zero. For example, regardless of the characteristic, a smooth cubic surface   can contain at most eighteen Eckardt points---unless its defining equation achieves the minimal F-pure threshold  of bound (\ref{Lowerbound2}), in which case this {\it extremal} cubic surface  contains exactly {\it forty-five } Eckardt points---one in each and every  of the forty-five tri-tangent planes   \cite{cubicspaper}.  As another example, in characteristic zero (or when $p>d$),
the maximal number of lines on a smooth projective surface of degree $d$ is bounded above by a quadratic function in $d$\  \cite{Segre.43, Rams+Schutt.15-64lines, Bauer+Rams}; however, the number of lines on a smooth  projective {\it extremal} surface---one whose defining polynomial achieves the lower bound (\ref{Lowerbound2})--- is {\it quartic} in its degree  \cite[3.2.3]{Brosowsky+Page+Ryan+Smith}.  These examples  confirm    that  singularities can be {\it much worse} 
   in characteristic $p$ than in characteristic zero,  reflecting  the fact that our  sharp lower bound on F-pure threshold  is much smaller than corresponding bounds on log canonical threshold; see paragraph (\ref{EM}).

In the second half of the paper, we tackle the classification of the extremal forms achieving the lower bound (\ref{Lowerbound2}).
Theorem \ref{main1} says that they are in  high "Frobenius powers" of the unique homogeneous maximal ideal---more precisely,  we can represent them by 
$$
x_1^{q}L_1 + \dots + x_n^{q}L_n, 
$$
where the $L_i$ are  linear forms. We call such polynomials {\bf Frobenius  forms}.
Frobenius forms have a convenient matrix representation, not unlike the matrix representation for a quadratic form,  
that we can exploit to  prove:
 \begin{theorem}\label{fullrank2}
Over any  algebraically closed field  $k$ of positive characteristic $p$, there is a bijection between the projective equivalence classes of non-degenerate Frobenius forms in $n$ variables (of fixed degree $q+1$) and partitions of $n$. Only  one of these defines an isolated singularity, namely the one represented by the "diagonal" form $x_1^{q+1}+x_2^{q+1}+ \dots + x_n^{q+1}$.
\end{theorem}

In three variables, for example, we have three distinct projective equivalence classes of non-degenerate extremal singularities (over each allowable field and of each allowable degree), corresponding to the three partitions $3$, $2+1$ and $1+1+1$.  For example, Theorem  \ref{fullrank2} implies that there are precisely five extremal {\it cubic} surfaces up to change of coordinates,  not including the degenerate ones (which  are cones over extremal curves); 
these extremal cubic surfaces  turn out to be exactly the non-Frobenius split cubic surfaces, a class studied in \cite{cubicspaper}.   The case where $n= 5$ was handled in  \cite{WicaPaper}.

Theorem \ref{fullrank2} is a combination of Theorems \ref{fullrank} and \ref{partitions} proved in Sections 6 and 7, respectively. 
We believe that the classification of Frobenius forms by partitions, Theorem \ref{partitions},  is new, though it resembles a {\it different} classification problem considered by Hasse and Witt \cite{HesseWitt} (see Remark \ref{HWCompare}). Theorem \ref{fullrank},   that there is only one  (up to change of coordinates) 
with  isolated singularity,  
 follows also from a theorem of Beauville when there are at least $4$ variables \cite{beauville}; we give a straightforward  argument using Hilbert's Nullstellensatz.

\medskip
 \subsection{Comparison to Hermitian Forms in Prime  Characteristic.}\label{Herm}
Frobenius forms should not be confused with characteristic $p$ {\it Hermitian forms}, which form a very special subclass of Frobenius forms. Hermitian forms have  long  been known  to be 
  extremal with respect to the number of rational points the corresponding projective hypersurface  contains;  see, for example,  \cite{segre}, \cite{BC},  \cite{HommaKim}.  The existence of many rational points is a consequence of the many linear spaces contained in the projective hypersurface defined by a Frobenius form; see the discussion in \cite[\S35]{kollar14} for example.
 The classification of Hermitian forms is much simpler than general Frobenius forms: there is only one of each rank  \cite{BC}\cite{HesseWitt}. For the definition of Hermitian form,  see Remark \ref{Hermitian}.

\medskip
 \subsection{Comparison to Lower Bounds for Log Canonical threshold}\label{EM}
 Lower bounds on log canonical thresholds have been studied in  \cite{CheltsovPark},   \cite{deFernexEinMustata}, and  \cite{DemailyPham}. 
Our  lower bound on the $F$-pure threshold  immediately implies an analogous bound on the log canonical threshold of a complex hypersurface  by reduction to characteristic $p$; see \S \ref{charzero}.  However,  while sharp in prime characteristics, the corresponding bound for the log canonical threshold is far from sharp. This is to be expected: 
singularities in prime characteristic can be "bad" in ways not possible over $\mathbb C$. 

For example,  de Fernex,  Ein, and Musta\c{t}\u{a} \cite{deFernexEinMustata} prove that for a complex
homogeneous polynomial   of degree $d$  in $n$ variables, the log canonical threshold is bounded below by 
$$
\min\left(\frac{n-r}{d}, 1\right),
$$
where $r$ is the dimension of the singular locus of the corresponding affine hypersurface. 
The corresponding statement in positive characteristic, however, is spectacularly false. For example, 
the polynomial in characteristic $p$ defined by 
$$
x_1^{p^e+1} +  x_2^{p^e+1}  + \dots  + x_n^{p^e+1} 
$$
has $F$-pure threshold $\frac{1}{p^e}$ (a simple case of Theorem \ref{main1}), which is much smaller, for $n\gg p$, than  the value provided by the de Fernex-Ein-Musta\c{t}\u{a} bound.

 \medskip
 \noindent
{\sc Acknowledgements.} The authors thank  Karl Schwede, Mircea Musta\c{t}\u{a}, and J\'anos
 Koll\'ar for their  interest, and  especially Koll\'ar for directing us to the references \cite{kollar14} and \cite{HesseWitt}, and Manuel Blickle for help translating \cite{HesseWitt}.    The fifth author also gratefully acknowledges lively discussions with Damiano Testa, where she first learned of some of the extremal properties of certain  forms, such as the inseparability of their Gauss maps. 
\section{Background on the $F$-pure Threshold}

Fix  a  field $k$ of  positive characteristic  $p$.
Given a homogeneous form $f\in k[x_1, \dots, x_n]$, we  consider the singularity, at the origin, of the affine  hypersurface  defined by vanishing of $f$. 
Algebraically stated, we consider the singularity of $f$ at $\mathfrak m$,  where $\mathfrak m$ denotes the maximal ideal  $\langle x_1, \dots, x_n\rangle$. 

\begin{definition}\label{FPureDef} The $F$-pure threshold of $f\in k[x_1, \dots, x_n]$ (at the maximal ideal $\mathfrak{m} $) is the real number
$$
\fpt(f) = \sup \left\{\frac{N}{p^e}  \,\,\, \bigg| \,\,\, f^N \notin \mathfrak{m}^{[p^e]} \right\} =  \inf \left\{\frac{N}{p^e}  \,\,\, \bigg| \,\,\, f^N \in \mathfrak{m}^{[p^e]} \right\},
$$
where  $\mathfrak{m}^{[p^e]} $ denotes the Frobenius power $\langle x_1^{p^e}, \dots, x_n^{p^e}\rangle$ of $\mathfrak m$.
\end{definition}

While first explicitly   defined  (in a more general setting) by Takagi and Watanabe  \cite{takagi+watanabe.F-pure_thresholds} as the "threshold $c$" beyond which the pair $(S, f^{c})$ 
fails to be $F$-pure (see  \cite{HaraYoshida}),   the definition above  is a reformulation that has evolved through the work of many authors 
(e.g., see    \cite{mustata+takagi+watanabe.F-thresholds}, \cite{blickle+mustata+smith.discr_rat_FPTs}).  A  gentle introduction can be found in the survey \cite{benito+faber+smith.measuring_sing_with_frob}.

\medskip
Although not obvious, the $F$-pure threshold is in fact a {\it rational number} \cite{blickle+mustata+smith.discr_rat_FPTs}. 
Further basic  properties,   including some  immediate upper and lower bounds well-known to experts, are summarized  below in the  setting we will need them: 

\medskip
\begin{proposition}  Let $f$ be a homogeneous form of degree $d>0$ over a field $k$ of characteristic $p>0$.  Then 
\begin{enumerate}\label{easy}
\item  $\fpt(f) \leq  1$.
\item 
For any $r\geq 1$, we have $\fpt(f^r)= \frac{\fpt(f) }{r}.
$
\item  $\fpt(f) \geq \frac{1}{d}$, with equality when $f$ is a power of a linear form; \cite[4.1]{takagi+watanabe.F-pure_thresholds}.
\item If $f$ is in $n$ variables, then  $\fpt (f) \leq  \frac{n}{d}$.
\end{enumerate}
\end{proposition}

\begin{proof}
For (1), since $f^1\in \mathfrak m^{[p^0]} = \mathfrak m$, $1\in   \left\{\frac{N}{p^e}  \,\,\, \bigg| \,\,\, f^N \in \mathfrak{m}^{[p^e]} \right\}$.  So the infimum of this set  is at most one, always. The second statement similarly  follows easily from the definition.

For (3), we argue as follows.  Let $r_e = {\lfloor p^e/d \rfloor -1}$. Then $f^{r_e}$ has degree less than $p^e$, so is not in $\mathfrak m^{[p^e]}$. So $\frac  {\lfloor p^e/d \rfloor -1}{p^e} \in  \left\{\frac{N}{p^e}  \,\,\, \bigg| \,\,\, f^N \notin \mathfrak{m}^{[p^e]} \right\}$ for all $e$, and so the supremum is at least 
$ \frac  {\lfloor p^e/d \rfloor -1}{p^e} $ for all $e$. Since these converge from below to $\frac{1}{d}$ as $e$ goes to infinity, the $F$-pure threshold is at least $\frac{1}{d}$.

For (4), observe that 
$
\mathfrak m^m \subset \mathfrak m^{[p^e]}
$ 
for any $m\geq np^e-n+1$.
So  if $f$ has degree $d$, then
for any natural number  $r_e$ such that $r_e\geq \frac{np^e-n+1}{d}$, we have
$$f^{r_e} \in \mathfrak m^{[p^e]}.$$
In particular, this is true for  $$r_e = \left \lceil   \left(\frac{np^e-n+1}{d} \right) \right \rceil =    \left\lceil p^{e} \left( \frac{n}{d} - \frac{n}{dp^e}  + \frac{1}{dp^e} \right) \right\rceil  .$$
This means that 
$$\frac{r_e}{{p^e}}  \in 
  \left\{\frac{N}{p^e}  \,\,\, \bigg| \,\,\, f^N \in \mathfrak{m}^{[p^e]} \right\}
  $$
  for all $e. $  Since  $\frac{r_e}{{p^e}}$ converges to $\frac{n}{d}$ from above, the infimum of this set is at most  $\frac{n}{d}$. 
\end{proof}

\bigskip
The idea is that "worse" singularities have smaller $F$-pure threshold, just as smaller log canonical thresholds for a complex singularity indicate "worse singularities." 

 \subsection{Connection with log canonical threshold}\label{charzero}
The  {\it log canonical threshold} is an invariant of a complex singularity that can be defined using either integrability or Hironaka's resolution of singularities. For a $\mathbb Q$-divisor $D$ on a smooth complex variety $X$, it is the threshold beyond which the pair $(X, cD)$ fails to be log canonical; see, for example,   \cite{kollar}.

A form  $f$ over the complex numbers determines a collection of forms over  fields of characteristic $p$, for varying  $p$, as follows. Let $A$ be the finitely generated subring of $\mathbb C$ obtained by adjoining all the complex coefficients of $f$ to $\mathbb Z$. Interpreting the form $f$ as an element of $A[x_1, \dots, x_n]$, for each  $\mu\in \, \operatorname{maxSpec} A$, 
we define  $f_{\mu}$ to be  the image of $f$ in the quotient ring $ A/\mu[x_1, \dots, x_n]$, a polynomial ring over the  finite field $ A/\mu$. 

\medskip
The work of Hara and Yoshida \cite{HaraYoshida} and
 Takagi and Watanabe \cite{takagi+watanabe.F-pure_thresholds} implies that the log canonical threshold of $f$ is
\begin{equation}\label{lct}
\sup \left\{ \fpt(f_{\mu}) \,\, | \,\, \mu \in \, \operatorname{maxSpec} \, A\right\}.
\end{equation}
In particular,  any lower bound on the $F$-pure threshold (independent of $p$) implies one for the log canonical threshold.  In particular, our Theorem \ref{main} implies that the 
 log canonical threshold of a reduced complex form of degree $d$ is bounded below by $\frac{1}{d-1}.$ 
 However, this bound is far from sharp; see paragraph \ref{EM}. 
 
 \medskip
On the other hand, (\ref{lct}) implies that upper bounds for the log canonical threshold suggest tight upper bounds for the $F$-pure threshold. 
For example, it is easy to compute that a degree $d$ form over $\mathbb C$ defining an isolated singularity has log canonical threshold $\min(\frac{n}{d}, 1)$, and indeed,  
 the proof of Proposition \ref{easy}  (4)  above shows that  $\min(\frac{n}{d}, 1)$ is what we might expect for the $F$-pure threshold of most general  forms  $f$ of degree $d$ in $n$ variables. This  intuition  is made precise in \cite{Hernandez}.
 
  One research thread in the literature is concerned with understanding the extent to which the $F$-pure threshold pushes up against these theoretic {\it upper} bounds. 
 A long-standing open conjecture predicts that when  $f$ is obtained by reduction to characteristic $p$ from a polynomial  over $\mathbb C$, then  for infinitely many $p$, the $F$-pure threshold  will be equal to the log canonical threshold of the corresponding complex singularity; see, for example, 
 \cite[3.6]{mustata+takagi+watanabe.F-thresholds}.
 
 \medskip

 In this paper, we tackle the opposite question: find general {\it  lower bounds}  on the $F$-pure threshold in terms of degree and investigate the extent to which we push up against these bounds.


\section{Lower Bounds on the $F$-pure Threshold}\label{LowerBounds}

 In this section, we establish the lower bound of our main theorem, Theorem \ref{main1},  by proving the following:

 \begin{theorem}\label{main}
 Fix an   arbitrary  field $k$ of characteristic $p>0$.
Let $f\in k[x_1, \dots, x_n]$ be a homogeneous geometrically reduced polynomial of degree  $d=\mathrm{deg}(f)\geq 2$.
Then 
\begin{equation}\label{Lowerbound}
\fpt(f) \ge \frac{1}{d-1}
\end{equation}
Furthermore,  if equality  holds in  (\ref{Lowerbound}), then   $d=q+1$, where $q\geq 1$ is a power of $p$.
\end{theorem}

\medskip

Having degree $p^e+1$ is  necessary but not sufficient to achieve the lower bound (\ref{Lowerbound}) on the $F$-pure threshold: some forms of degree $p^e+1$ are more singular than others. Characterizing these "maximally singular" forms where  the lower bound is achieved is the task of the next section; see Theorem \ref{justify}. 

\medskip

\begin{remark} 
For the bound (\ref{Lowerbound}), a careful reading of the proof  shows that  our hypothesis can be weakened from "reduced" to "not a power of a linear form." 
But for powers of linear forms, of course, the lower bound is false: the $F$-pure threshold of $x^d$ is $\frac{1}{d}$. 

For the subsequent statement about what happens when the  lower bound is achieved,  however, we need the "reduced" hypothesis. For example, 
$x^6y$ has $F$-pure threshold $\frac{1}{6} = \frac{1}{d-1}$ in every characteristic, but $6$ is not a power of $p$. 
\end{remark}

\smallskip
\begin{remark}
The main theorem of \cite{BhattSingh} can be interpreted to give a lower bound of $1-\frac{d-2}{p}$ on the $F$-pure threshold in the very special case where the degree $d$ of the form $f$ is equal to  the number of variables $n$, the projective hypersurface defined by $f$ is smooth, and the characteristic $p>n=d$. (See also \cite{Mueller1} for the quasi-homogeneous case.)  However, because of the restriction that $p>d$, the Bhatt-Singh bound never applies in our extremal case.  
\end{remark}

\medskip
Before beginning the proof of Theorem \ref{main}, we need a few lemmas.
The first allows us to assume the ground field is algebraically closed.

\begin{lemma}\label{extend}
Let $k\subset k'$ be any field extension, of  characteristic  $p>0$. For any $f\in k[x_1, \dots, x_n]$, the $F$-pure threshold is independent of whether we view $f$ as a polynomial over $k$ or over $k'$.
\end{lemma}

\begin{proof}
Let  $ A = k[x_1, \dots, x_n]$ have homogeneous maximal ideal $\mathfrak m_A$ and $ B = k'[x_1, \dots, x_n]$ have homogeneous maximal ideal $\mathfrak m_B.$
Note that $A\subset B$ is faithfully flat so that $IB\cap A = I$ for any ideal $I$ of $A$.

Now, if $f^N \in \mathfrak m_A^{[p^E]},$ then the same is true in $B$. So the $F$-pure threshold over $A$ is at least the $F$-pure threshold over $B$.
But conversely, if  $f^N \in \mathfrak m_B^{[p^E]}= \mathfrak m_A^{[p^E]}B$,  then intersecting with $A$, we have   $f^N \in \mathfrak m_A^{[p^E]}B\cap A =  \mathfrak m_A^{[p^E]}$. This gives the reverse inequality. 
\end{proof}

The next is a codification of a well-known trick  we use many times, whose statement we make explicit for the convenience of the reader. The proof we leave as an exercise.

\begin{lemma}\label{reg} Let $y_1, \dots, y_n$ be a regular sequence in a commutative ring.  Suppose there exists an element  $g$, and natural numbers $a_i\leq N_i$ such that  
$$(y_1^{a_1}y_2^{a_2}\dots y_n^{a_n})g \in \langle y_1^{N_1},  y_2^{N_2}, \dots, y_n^{N_n}\rangle.$$
Then  $g \in  \langle y_1^{N_1-a_1},  y_2^{N_2-a_2},\, \dots,\, y_n^{a_N-a_n}\rangle.$
\end{lemma}

\bigskip 
The next two lemmas will be used to reduce
Theorem \ref{main} to the case of two variables.

\medskip

\begin{lemma}\label{dumbbound} Let $f$ be a homogeneous polynomial  in $k[x_1,\dots, x_n]$. For any linear form $L$ not dividing $f$, let 
$\bar{f}$ denote the image  of $f$  in $k[x_1, \dots, x_n]/\langle L \rangle \cong k[x_1, \dots, x_{n-1}]$ .
Then
$$ \fpt (f) \geq   \fpt (\bar{f}).$$ 
\end{lemma}

\begin{proof}
Suppose $f^N\in \mathfrak m^{[p^e]}$. Then also  ${\bar f}^N\in  \bar{\mathfrak m}^{[p^e]}$. So we have an inclusion of sets
 $$ \left\{\frac{N}{p^e}  \,\,\, \bigg| \,\,\, f^N \in \mathfrak{m}^{[p^e]} \right\} \subseteq
  \left\{\frac{N}{p^e}  \,\,\, \bigg| \,\,\, {\bar f}^N \in  \bar{\mathfrak m}^{[p^e]} \right\}.$$
  So the infimum of the left-hand set is at least the infimum of the right-hand set. That is, $ \fpt (f) \geq   \fpt (\bar{f})$.\end{proof}

\medskip
\begin{remark}
Lemma \ref{dumbbound} can be used to give a quick proof of the easy lower bound $\fpt(f)\geq \frac{1}{d}$, where $d=\deg(f)$,   shown Proposition \ref{easy}. 
Namely, by modding out $n-1$ linearly independent one-forms, we eventually reduce to the one-variable case, and have $\fpt(f)\geq \fpt(x^d) =  \frac{1}{d}$.
\end{remark}

\medskip

Finally, we need the following Bertini-type theorem for reduced varieties:
\begin{lemma}\label{bertini} Consider a polynomial ring $k[x_1, \dots, x_n]$ in at least three variables over an infinite field. 
Fix a reduced form $f$ in   $k[x_1, \dots, x_n]$, and  for any linear form $L$, let $\bar{f}$ denote the image of $f$ in the quotient ring $k[x_1, \dots, x_n]/\langle L\rangle \cong k[x_1, \dots, x_{n-1}].$
 If $L$ is sufficiently general, then the polynomial 
$ \bar{f}$ is also reduced in the polynomial ring $k[x_1, \dots, x_{n-1}]$. 
\end{lemma}

\begin{proof}
This is essentially a restatement of \cite[3.4.14]{Flenner-OCarroll-Vogel}, which implies that (over an infinite field) a general hyperplane section of a reduced scheme
$X\subset \mathbb P^{n-1}$ is reduced: because $f$ is reduced, the projective variety  defined by $f$ in $\mathbb P^{n-1}$ is reduced, so also the general hyperplane section---namely, the variety defined by 
$f$ and $L$---is reduced. Because the homogeneous coordinate ring of this hyperplane section is isomorphic to $k[x_1, \dots, x_n]/\langle L, f\rangle \cong k[x_1, \dots, x_{n-1}]/\langle \bar{f}\rangle,$ we see that also $\bar{f}$ is reduced. 
\end{proof}

\begin{proof}[Proof of Theorem \ref{main}]  By Lemma \ref{extend}, we may assume the ground field is algebraically closed.

We first 
reduce to the case of  two variables.   Given a form  $f$ of degree $d$ in $n$ variables, we can successively mod out a sequence of $n-2$ of independent  linear forms until we have a  form $\bar f$ of degree $d$ in two variables. Now Lemma \ref{dumbbound} 
implies that any lower bound on the $F$-pure threshold of $\bar f$ is also a lower bound for $f$. Likewise, if this lower bound is achieved for 
$f$, then it is also achieved for 
$\bar{f}$.  So because $d$ and $p$ are the same for $f$ and $\bar f$, it suffices to prove the required implication for $\bar f$.
Finally,  if $f$ is reduced,   Lemma \ref{bertini} ensures that $\bar f$ is reduced, by choosing a sufficiently general sequence of linear forms. Thus  the proof of Theorem \ref{main} reduces to the case of  two variables.

We next dispose of the case $d=2$. Any reduced polynomial of degree 2 in two variables factors into two distinct linear factors, so 
without loss of generality,  $f=xy$. By direct computation,  the   $F$-pure threshold is $1$ in every characteristic. So equality holds in (\ref{Lowerbound}) and  $d=p^0+1$.

Now assume $d\geq 3$.  Given a form $f$ of degree $d$ in two variables, we can factor as  $f=xyg$ where $g$ is a form of degree $d-2$. To prove the lower bound (\ref{Lowerbound}), it suffices to show  that 
 for any $p^E$ such that $f^N \in \mathfrak{m}^{[p^E]}$, we have 
\begin{equation}\label{des}
  \frac{N}{p^E} \ge \frac{1}{d-1}.
\end{equation}
But if $f^N \in \mathfrak{m}^{[p^E]} = \langle x^{p^E}, y^{p^E}\rangle,$ then writing 
$$
(xyg)^N  = Ax^{p^E} + B y^{p^E}
$$
for some homogeneous $A$ and $B$, we can use the fact that $\{x, y\}$ is a regular sequence  (Lemma \ref{reg}) to see that 
$$ 
g^N \in \langle x^{p^E-N}, y^{p^E-N}\rangle.
$$
By comparing degrees, it follows that $N(d-2)\ge p^E-N$, which is equivalent to the desired inequality (\ref{des}). This shows the $F$-pure threshold is at least $\frac{1}{d-1}$.

\medskip 

We now investigate what happens when equality holds in (\ref{Lowerbound}).
Assume that $\displaystyle \fpt(f)=\frac{1}{d-1}$. For all $e\geq 1$, we have
$$
 \displaystyle \frac{\lceil \frac{ p^e+1}{d-1} \rceil}{p^e}  \geq \frac{p^e+1}{p^e(d-1)} > \frac{1}{d-1 }=  \fpt(f).
 $$
Since $\fpt(f)$ is the supremum of the set $\{\frac{N}{p^E} \, | \, f^N \not\in \mathfrak m^{[p^E]}\}$, and  $  \frac{\lceil \frac{ p^e+1}{d-1} \rceil}{p^e} $ is
strictly  bigger than this supremum, it must be that 
$$
f^{\lceil \frac{ p^e+1}{d-1} \rceil}\in \mathfrak m^{[p^e]}.
$$ 
To ease notation, we set $K_e = \lceil \frac{ p^e+1}{d-1} \rceil$. We have
$
f^{K_e} \in \langle x^{p^e}, y^{p^e}\rangle,
$
so we write 
\begin{equation}\label{eq5}
f^{K_e}  = A x^{p^e} + B y^{p^e}
\end{equation}
where $A$ and $B$ are forms of degree  $ dK_e-p^e$.

Also using the strict inequality  $\frac{K_e}{p^e} > \frac{1}{d-1}$, we have  $ \frac{K_e}{p^e} - \frac{1}{p^E} > \frac{1}{d-1}$ for all sufficiently large $E$.
So similarly,
$$
f^{p^{E-e}K_e-1}  \in \langle x^{p^E}, y^{p^E}\rangle,  
$$
and we can write 
\begin{equation}\label{eq6}
f^{p^{E-e}K_e-1}  = C x^{p^E} + D y^{p^E}
\end{equation}
for some forms $C$ and $D$.

Now raising (\ref{eq5}) to the power  $p^{E-e}$ and multiplying (\ref{eq6}) by $f$, we have two different expressions for the form $f^{p^{E-e}K_e}$. Subtracting them, we have
\begin{equation}\label{7}
(A^{p^{E-e}} - fC) x^{p^E} + (B^{p^{E-e}} -f  D)  y^{p^E} = 0.
\end{equation}
Again using the fact that $x, y$ is a regular sequence, we conclude that
\begin{equation}\label{eqA}
(A^{p^{E-e}} - fC) \in \langle  y^{p^E} \rangle.
\end{equation}
But we claim that 
\begin{equation}\label{8}
\deg (A^{p^{E-e}} - fC)  = p^{E-e}\deg A  <p^E,
\end{equation}
which implies that  $A^{p^{E-e}} - fC = 0$.

To check  claim (\ref{8}),
recall that the degree of $A$ is
$dK_e-p^e$, so that  (\ref{8}) is equivalent to $dK_e < 2 p^e$.
In turn, we have
$$
dK_e = d \left\lceil{{ \frac{p^e+1}{d-1}   }  }\right\rceil  \,\,
 \leq  \,\, d  \left( \frac{p^e+1}{d-1} + 1 \right),
$$
which is  less than $2p^e$ for all large   $e$.

Having established the veracity of claim (\ref{8}) for $E\gg e\gg  0$, we 
can conclude using the inclusion in (\ref{eqA}) that $A^{p^{E-e}} - fC = 0$, so that
$$
A^{p^{E-e}}  =  fC.
$$
We now invoke the fact that  $f$ is a product of {\it distinct} irreducible polynomials:  the unique factorization property of the polynomial ring implies that  $f$ must divide the form $A$.   Similarly, $f$ divides $B$.

 Now, because $f$ divides both $A$ and $B$, we can divide $f$ out of both sides of   equation (\ref{eq5}) above, to get
\begin{equation}\label{eq9}
f^{K_e-1} \in \langle x^{p^e}, y^{p^e}\rangle.
\end{equation}
 Remembering that $f=xyg$, where $g$ has degree $d-2$, we can again  use the fact that $x, y$ is a regular sequence (Lemma \ref{reg}) to deduce that 
 \begin{equation}\label{eq8}
 g^{K_e-1} \in  \langle x^{p^e-K_e+1}, y^{p^e-K_e+1}\rangle.
 \end{equation}
 Looking at degrees, this says that 
 $$
 (K_e-1)(d-2) \geq p^e-K_e+1,
 $$
which is equivalent to 
$$
(d-1)K_e \ge p^e +  d-1,
$$
or equivalently, 
\begin{equation}\label{inn}
K_e \ge \frac{p^e+d-1}{d-1}=\frac{(p^e+1)+(d-2)}{d-1} = \frac{p^e+1}{d-1}   +  \frac{d-2}{d-1}.
\end{equation}
Remembering that  $K_e = \lceil \frac{ p^e+1}{d-1} \rceil$, we see that
 inequality (\ref{inn})   can 
 hold only  if  $\displaystyle \frac{p^e+1}{d-1}$  is {\bf not} an integer  {\bf and}
$$
\frac{p^e+1}{d-1} \ \ \ \ \mathrm{and}\ \ \ \ \frac{p^e+1}{d-1}   +  \frac{d-2}{d-1}
$$ round up to the same integer. 
This means that $\displaystyle \frac{p^e+1}{d-1}$ must be equal to 
$\displaystyle \left \lfloor \frac{p^e+1}{d-1} \right\rfloor  + \frac{1}{d-1}$; put differently,  the remainder when we divide $p^e+1$ by $d-1$ is 1. 
So     $d-1$ divides $p^e$. In this case,  $d-1$ is a power of $p$ (as desired).
 The proof is complete.
\end{proof}

\section{Extremal Singularities}\label{Equality}

In this section, we complete the proof of Theorem \ref{main1} by characterizing  those reduced homogeneous polynomials for which the lower bound
$$
\fpt(f) \geq \frac{1}{\deg f -1}
$$
is achieved. 
\begin{definition}\label{defExtremal}
A geometrically reduced form $f$ of degree $d>2$  is called an {\bf extremal singularity} if its $F$-pure threshold is equal to $\frac{1}{d-1}$. 
\end{definition}

\begin{remark}
 We exclude the case $d=2$ from Definition {\ref{defExtremal}} because the  $F$-pure threshold of a geometrically reduced quadratic  is one in every characteristic. So not much interesting is happening:
quadratic forms achieve both  the theoretical lower bound and the theoretic upper bound on the $F$-pure threshold (Proposition \ref{easy}) in every case---no finer gradation of singularities exists. 
\end{remark}

The next theorem characterizes extremal singularities in terms of their equations:

\medskip
\begin{theorem}\label{justify}
Let $f$ be a geometrically reduced form of degree $d$ over a field of positive characteristic $p$. Then the  $F$-pure threshold of $f$ is $\frac{1}{d-1}$ if and only if $f$ can be written 
\begin{equation}\label{FF}
x_1^{p^e} L_1+ x_2^{p^e} L_2 + \dots + x_n^{p^e} L_n
\end{equation}
for some $e\geq 0$,  where the $L_i$ are linear forms. 
\end{theorem}

The proof of Theorem \ref{justify} relies on the following lemma  justifying the intuition that polynomials in  "Frobenius powers" $\mathfrak m^{[p^e]} = \langle x_1^{p^e}, \dots, x_n^{p^e}\rangle$  are "more singular"   than polynomials not in $\mathfrak m^{[p^e]}$:

\medskip

\begin{lemma} \label{firstbound}The $F$-pure threshold of $f$ is less than or equal to  $\frac{1}{p^e}$ if and only if  $f\in \frak m^{[p^e]}$.
In fact, if  $f\not\in \frak m^{[p^e]}$, then  $\fpt(f) \geq \frac{1}{p^e} + \frac{1}{p^{2e}}$. 
\end{lemma}

\begin{proof} First assume  $f\in \frak m^{[p^e]}$.  Then $\frac{1}{p^e}$ is in the set $ \left\{\frac{N}{p^E}  \,\,\, \bigg| \,\,\, f^N \in \mathfrak{m}^{[p^E]} \right\}$.
So the $F$-pure threshold, which is the {\it infimum} of this set, is bounded above by  $\frac{1}{p^e}$.

For the converse statement, first observe that 
$S= k[x_1, \dots, x_n]$ can be viewed as a  free module over the subring $S^{p^{e'}}$ for all $e'$. Indeed, we can take $\{\lambda x^m\}$ as a basis, where $\lambda$ ranges over a basis for $k$ over $k^{p^{e'}}$ and $x^{m}$ ranges through all monomials in the $x_i$. 
Moreover,  if 
$f\notin \mathfrak m^{[p^{e'}]}$, then it can be taken to be a part of a free basis   for $S$ over $S^{p^{e'}}$---for example, taking any monomial $x^{m'}$  that appears in $f$ with all exponents less than $p^{e'}$, we can replace $x^{m'}$ in the basis  $\{\lambda x^m\}$ by $f$ to get another basis. 
In particular,  projection onto the  $S^{p^{e'}}$-submodule generated by $f$ gives us an  $S^{p^{e'}}$-linear map $\pi:S\rightarrow S^{p^{e'}}$ sending $f$ to 1.   

Now, assume  $f\notin \mathfrak m^{[p^e]}$.   Then for all $e'$, the flatness of Frobenius \cite{Kunz} implies that also $f^{p^{e'}}\notin \mathfrak m^{[p^{e+e'}]}$. Furthermore, for all $e'\geq e$, we have
$$
f^{p^{e'}+1} =  f \cdot f^{p^{e'}}\notin \mathfrak m^{[p^{e+e'}]},
 $$ for otherwise, we could apply $\pi$ from the previous paragraph to see that $ f^{p^{e'}}\in \mathfrak m^{[p^{e+e'}]}$.
 This means that the rational number $\frac{p^{e'}+1}{p^{e+e'} }$ is in the set $  \left\{\frac{N}{p^e}  \,\,\, \bigg| \,\,\, f^N \notin \mathfrak{m}^{[p^e]} \right\}$ for all $e'\geq e$.
 So the supremum of this set is at least $\frac{1}{p^e} + \frac{1}{p^{e+e'}}$ for all $e'\geq e$. The largest of these is when $e=e'$, so the supremum is at least  $\frac{1}{p^e} + \frac{1}{p^{2e}}$. This completes the proof. 
\end{proof}

\begin{proof} [Proof of Theorem \ref{justify}]
Both directions follow  from facts we have already established. 

First, we claim that the  form $f = x_1^{p^e} L_1+ x_2^{p^e} L_2 + \dots + x_n^{p^e} L_n$  has $F$-pure threshold $\frac{1}{d-1} = \frac{1}{p^e}$. Indeed, $\fpt(f) \geq \frac{1}{p^e}$ by Theorem \ref{main},  whereas $\fpt(f) \leq \frac{1}{p^e}$ by Lemma \ref{firstbound}.

\medskip
For the other direction, suppose that  $\fpt(f) = \frac{1}{d-1}$ for some  reduced form $f$. Now Theorem \ref{main} tells us  that $d=p^e+1$ for some $e$, which means that   $\fpt(f) = \frac{1}{d-1} = \frac{1}{p^e}$ for some $e$.  Now Lemma \ref{firstbound} guarantees that $f\in \frak m^{[p^e]}$. Thinking about degrees, we see that $f$ must be of the form (\ref{FF}). 
\end{proof}


\section{Matrix Representation of Frobenius Forms}\label{LinearAlgebra}

Our next goal is to study the forms which achieve the minimal possible F-pure threshold according to the bounds proved in the previous sections. 
 Theorem \ref{justify}  tells us that such polynomials  have the  special form (\ref{FF}), which warrants a name: 
\medskip

\begin{definition} 
A {\bf Frobenius form} is a form of  degree $q+1$ in the ideal $\mathfrak m^{[q]} = \langle x_1^q, \dots, x_n^q\rangle$, where $q$ is a positive power of the characteristic, $p$. Equivalently, a Frobenius form is a polynomial that can be written 
\begin{equation}\label{FF2}
 \sum x_i^{p^e} L_i,
 \end{equation}
where the $L_i$ are linear forms and $e>0$. 
\end{definition}

\bigskip
The formula  (\ref{FF2}) implies that a Frobenius form can be uniquely  factored as
\begin{equation}\label{matrix}
 h = \begin{bmatrix}   x_1^{p^e} &  x_2^{p^e}  & \hdots & x_n^{p^e}\end{bmatrix}
 A \begin{bmatrix}  x_1 \\ x_2 \\ \vdots  \\ x_n \end{bmatrix},
 \end{equation}
  where $A$ is the  $n\times n$ matrix whose $i$-th row is made up of the coefficients of the linear form $L_i$.  This allows us to use linear algebra to 
 conveniently study Frobenius forms. 
  \bigskip

Consider how changing coordinates affects the matrix representing a Frobenius form. 
 For a matrix $B$ of any size,  we denote by $B^{[p^e]}$ the matrix obtained by raising all entries to the $p^e$-th power. If $g$ is a change of coordinates represented by an invertible $n\times n$ matrix,  then 
 $$g \cdot  \begin{bmatrix}   x_1^{p^e}  \\  x_2^{p^e}   \\  \vdots   \\ x_n^{p^e} \end{bmatrix} = g^{[p^e]} \begin{bmatrix}  x_1^{p^e}  \\  x_2^{p^e}   \\ \vdots \\ x_n^{p^e} \end{bmatrix} = [g 
   \begin{bmatrix}  x_1 \\ x_2 \\ \vdots \\ x_n \end{bmatrix} ]^{[p^e]}.$$
  Here the notation $\cdot$ indicates the ring automorphism induced by the linear change of coordinates, and all other adjacent symbols are usual matrix product.

  So our change of  coordinates formula  for $g$ acting on $h$ is 
  $$
g \cdot \begin{bmatrix}  x_1^{p^e}  &  x_2^{p^e}   &  \hdots & x_n^{p^e} \end{bmatrix}
  A   \begin{bmatrix}  x_1 \\ x_2 \\ \vdots \\ x_n \end{bmatrix} = 
 \begin{bmatrix}  x_1^{p^e}  &  x_2^{p^e}   & \hdots & x_n^{p^e} \end{bmatrix}
\left[g^{[p^e]}\right]^{tr}  A g  \begin{bmatrix}  x_1 \\ x_2 \\ \vdots \\ x_n \end{bmatrix},
 $$
 where the superscript "$tr$" indicates the transpose. 
 We can write this in the compact form 
\begin{equation*}
g \cdot [( \vec{x}^{[p^e]})^{tr} A \,  \vec{x}]  = 
 (\vec{x}^{[p^e]})^{tr}  \left[g^{[p^e]}\right]^{tr}  A g \, \vec{x}.
\end{equation*}
 That is, if $h$ is a Frobenius form represented by the matrix $A$, then the Frobenius form $g\cdot h$, where $g$ is any linear change of coordinates, is represented by the matrix 
    $ \left[g^{[p^e]}\right]^{tr}  A g$. This action is {\it different} than some others that may be familiar to the reader; see 
    Remark \ref{HWCompare}.

   \medskip
 
 It is worth recording, for future reference,  how each elementary coordinate operation  affects  the matrix $A$ representing  a Frobenius form.
  \begin{lemma}\label{rowops}
 Let 
 \begin{equation}\label{matrix2}
 h = \begin{bmatrix}   x_1^{p^e} &  x_2^{p^e}  & \hdots & x_n^{p^e}\end{bmatrix}
 A \begin{bmatrix}  x_1 \\ x_2 \\ \vdots  \\ x_n \end{bmatrix},
 \end{equation}
 be a Frobenius form.  Then  elementary  linear changes of coordinates are reflected in $A$ as follows: 
\begin{itemize}
\item Swapping  two variables   ($x_i\leftrightarrow  x_j$),  fixing the others changes $A$ by swapping columns $C_i$ and $C_j$  {\bf and} rows $R_i$ and $R_j$,   fixing the others.
\item Multiplying coordinate $x_i$  by a non-zero scalar  $\lambda$ ($x_i\mapsto \lambda x_i$), fixing the others  changes $A$ by multiplying row $R_i$ by $\lambda^{p^e} $  {\bf and}  column $C_i$ by $\lambda$. 
\item Replacing $x_i$ by $x_i + \lambda x_j$ for some $j\neq i$, fixing the others changes $A$ by replacing column $C_j$ by column $C_j+\lambda C_i$   {\bf and}  row $R_j$ by row $R_j+\lambda^{p^e} R_i$.
\end{itemize}
 \end{lemma}
 
 \medskip
 \subsection{Embedding dimension,  rank and the singular locus}
 
 A form $f \in k[x_1, \dots, x_n]$ is {\bf non-degenerate} if it can't be written as a polynomial in fewer variables after any linear change of coordinates.
 In this case,  the singularity defined by $f$ has {\bf embedding dimension} $n$, meaning that the dimension of its Zariski cotangent space $\mathfrak m/\mathfrak m^2$ is $n$.    
  
  \medskip
  
The {\bf rank} of a Frobenius form is the rank of the representing matrix. The following proposition implies that the rank is the same as the co-dimension of the singular locus of the corresponding extremal singularity.

\medskip

\begin{proposition}\label{sing}
The singular locus of an extremal singularity defined by the Frobenius form $$
\begin{bmatrix} x_1^{p^e} & x_2^{p^e} & \cdots  & x_n^{p^e} \end{bmatrix} 
A  \begin{bmatrix} x_1\\ x_2 \\  \vdots  \\ x_n \end{bmatrix}
$$
is the $p^e$-fold  linear subvariety defined by the  equations
$$
 A^{tr}  \begin{bmatrix}
 x_1^{p^e} \\
 x_2^{p^e}\\
 \vdots \\
 x_n^{p^e}\\
 \end{bmatrix} = 0.$$
Put differentially,  the (reduced)  singular set is the linear space defined as the kernel of the matrix $ (A^{[1/p^e]})^{tr}$, where  $(A^{[1/p^e]}
)^{tr}$
is the transpose of the matrix whose entries are the $p^e$-th roots of the entries of $A$.
\end{proposition}

\begin{proof}
Write 
$$h = x_1^{p^e}L_1 + x_2^{p^e} L_2 + \dots + x_n^{p^e}L_n$$
where the coefficients of the  linear forms $L_i$ are given by the rows of $A = [a_{ij}]$. 
The singular locus is defined by the vanishing of the partial derivatives $\frac{\partial h}{\partial x_j}. $
But  for each  $j$, 
 $$\frac{\partial h}{\partial x_j} = x_1^{p^e} a_{1j} + \dots + x_n^{p^e}a_{nj} =
 \begin{bmatrix}
 a_{1j} & \cdots & a_{nj} 
 \end{bmatrix}
  \begin{bmatrix}
 x_1^{p^e} \\ \vdots \\  x_n^{p^e}
 \end{bmatrix} =  \left[
  \begin{bmatrix}
 a_{1j}^{\frac{1}{p^e}} & \cdots & a_{nj}^{\frac{1}{p^e}} 
 \end{bmatrix}
  \begin{bmatrix}
 x_1 \\ \vdots \\  x_n
 \end{bmatrix} \right]^{p^e}, $$ so the proposition follows. 
 \end{proof}

\medskip

Finally, we record a simple lemma which gives a nice form for a Frobenius form  in terms of its rank and the minimal number of variables in which it can be written.
\begin{lemma}\label{Janet}
A   Frobenius form of  rank $r$ can be written, in suitable coordinates, as 
$$
h = x_1^{p^e}L_1 + x_2^{p^e} L_2 + \dots + x_r^{p^e}L_r,
$$
where the $L_i$ are linearly independent linear forms. 
In this case,  if $h$ is non-degenerate, then its embedding dimension $n$ is
equal to the dimension of the space spanned by the forms $x_1, x_2, \dots, x_r, L_1, \dots, L_r$.  In particular,  $2r\geq n$. 
\end{lemma}

\begin{proof} Let $h$ be a non-degenerate Frobenius form in  $n$ variables, and let $A$ denote its matrix.  Swapping variables, assume the first $r$ rows of $A$ are linearly independent. Because the rows beyond the $r$-th are all dependent on the first $r$, a suitable sequence of row operations can be used to transform these bottom  rows into zero rows; the corresponding column operations (Lemma \ref{rowops}) do not affect these zero rows. Thus without loss of generality, we can assume  the bottom $n-r$ rows of $A$ 
 are zero rows. This implies that $h$ can be written as $x_1^{p^e}L_1 + x_2^{p^e} L_2 + \dots + x_r^{p^e}L_r$ for some linear forms $L_i$.
 The $L_i$ are linearly independent because their coefficient vectors span the row space of the matrix $A$, which has rank $r$.
 For the final statement, note that if  $x_1, x_2, \dots, x_r, L_1, \dots, L_r$ span a space of dimension less than $n$, then  $h$ can be written in fewer than $n$ variables, so it is degenerate. \end{proof}

  \begin{remark}\label{Hermitian}
  A very special kind of Frobenius form is a {\it Hermitian form} of characteristic $p$. These are Frobenius forms in which the matrix $A$ representing the form   satisfies $a_{ij}=a_{ji}^q$
  where $q$ is a power of $p$ for all $i,  j$.  In particular, since this implies $a_{ij}^{q^2} = a_{ij}$ for all $i, j$,  a Hermitian form is defined over the finite field $\mathbb F_{q^2}$ and  the Frobenius map ($q$-th power map) is an involution that  plays a role analogous to complex conjugation.    
   Hermitian hypersurfaces---projective hypersurfaces defined by Hermitian forms---have well-studied "extremal" properties, such as an abundance of rational points; see  \cite{BC},  \cite{segre},  and \cite{HommaKim}.
        \end{remark}

\section{Extremal Singularities of Full Rank}\label{FullRank}
In this section, we prove the following characterization of isolated extremal singularities.

 \smallskip

\begin{theorem}\label{fullrank}   Every  full rank   extremal  singularity over an algebraically closed field $k$ of characteristic $p>0$ 
 is represented, in suitable linear coordinates,  by the   diagonal form  $x_1^{q+1}+ \cdots + x_n^{q+1},$ where 
$q$ is some power of $p$. 
\end{theorem}

We prove Theorem \ref{fullrank} using  only  Hilbert's Nullstellensatz.
In the special case of a {\it Hermitian form}---that is, where the matrix satisfies $a_{ij} = a_{ji}^p$ for all $i, j$---Theorem \ref{fullrank} can be found\footnote{There appears to be some confusion in the literature interpreting the 1936 paper of Hasse and Witt
(see, e.g., the "warning and request" in [AH19]).}
 in 
\cite{HesseWitt} or  \cite{BC}.  When the embedding dimension is  at least four, Theorem \ref{fullrank}  follows from the main theorem of 
 \cite{beauville}.   Koll\'ar suggested an alternate proof as well; see Remark \ref{JK}.

\begin{proof}
We will prove this by induction on $n$. The case where $n=1$ is trivial. 

Let $h$ be a full rank Frobenius form  in $n$ variables with $n\geq 2$. 
Write $h$ as $$
h  = \begin{bmatrix} x_1^{p^e} & x_2^{p^e} & \cdots  & x_n^{p^e} \end{bmatrix} 
A  \begin{bmatrix} x_1\\ x_2 \\ \vdots \\ x_n \end{bmatrix}
$$
where $A$ is an $n\times n$ matrix over $k$.
The induction happens by showing that we can change  coordinates to put $A$ into the block form
\begin{equation}\label{form}
\begin{bmatrix}

* & * & \cdots  & * & 0 \\
*& *  & \cdots  &* & 0 \\
\vdots  & \vdots  & \ddots & \vdots  & \vdots  \\
*  & *  & \cdots  &* & 0 \\
0 & 0  & \cdots & 0  & 1\\
\end{bmatrix}
\end{equation}
Equivalently, this says we can write the Frobenius form as 
$$h = x_n^{p^e+1} + g(x_1, \dots, x_{n-1}),$$
where $g$ is  full rank Frobenius form  in the first $n-1$ variables. 
So if we know by induction that $g$ can be put into the desired form by a linear change of coordinates involving only the variables $x_1, \dots, x_{n-1}$, then it follows that $h$ is in  this form as well.

\medskip
We will use the following lemma:

\begin{lemma}\label{adela}
If a  full rank Frobenius form 
\begin{equation}\label{niceform}
h  = \begin{bmatrix} x_1^{p^e} & x_2^{p^e} & \cdots  & x_{n-1}^{p^e} & x_n^{p^e} \end{bmatrix} 
\begin{bmatrix}
* & * & \cdots & * & a_{1n}\\
* & * &  \cdots & * &   a_{2n} \\
\vdots & \vdots & \ddots & \vdots &\vdots \\
* &  * &  \cdots & * & a_{n-1,n} \\
a_{n1} & a_{n2}& \cdots & a_{n,n-1} & a_{nn} \\
\end{bmatrix}
 \begin{bmatrix} x_1\\ x_2 \\ \vdots \\ x_{n-1}   \\ x_n \end{bmatrix}
\end{equation}
satisfies 
 \begin{equation}\label{require}a_{in}a_{nn}^{p^e-1} = a_{ni}^{p^e} \ \ \ \mathrm{ for  \,\, all\  } i=1, \dots, n-1,\end{equation}
 then we can change coordinates to put $h$ in  the block form (\ref{form}).  That is, 
we can change coordinates to get $h$ in the form 
$$
x_n^{p^e+1} +  g
$$ where $g$ is a Frobenius form  in  $x_1, x_2, \dots x_{n-1}$. \end{lemma} 

\begin{proof} 
 Note that if $a_{nn} = 0$ is zero, then the condition (\ref{require}) implies that the last row is zero, contrary to the full rank assumption on $A$. 
So $a_{nn}\neq 0$. We can therefore assume, without loss of generality, that $a_{nn} = 1$. Indeed, 
scaling $x_n$ by a $(p^{e}+1)$-th root of $a_{nn}^{-1}$ (call it $c$)  changes the matrix  $A$ by multiplying row $n$ by $c^{p^e}$ and column $n$ by $c$ (see Lemma \ref{rowops}). 
This allows us to assume that $a_{nn} = 1$  without destroying condition (\ref{require}).

Now, assuming $a_{nn} = 1$,  the 
change of coordinates that sends 
 $$
x_n \mapsto  x_n -a_{n1}x_1-a_{n2}x_2 -  \cdots  - a_{n,n-1}x_{n-1} 
$$
and fixes $x_i$ for $1 \leq i \leq n-1$ gives us the desired form.  One simply checks that substituting $x_n -a_{n1}x_1-a_{n2}x_2 -  \cdots  - a_{n,n-1}x_{n-1} $ for $x_n$  into $h$ produces a polynomial of the form $x_n^{p^e+1} + g(x_1, \dots, x_{n-1})$. 
Alternatively, invoking Lemma \ref{rowops}, because of the special form of $A$, we see that subtracting $a_{ni}$ times column  $n$ from column $i$ will place a zero in the $i$-th column of the final row, while the corresponding row operation also makes the $i$-th row of the final column zero.
\end{proof}

\bigskip Continuing the proof of Theorem \ref{fullrank} now
armed with Lemma \ref{adela}, we note that it suffices to show that any full rank Frobenius form can be put in the form  (\ref{niceform}). 
Let $h = \begin{bmatrix} \vec x^{[p^e]} \end{bmatrix}^{tr}  A  \begin{bmatrix} \vec x \end{bmatrix} $ be an arbitrary Frobenius form.
Suppose   $g$ is a change of coordinate matrix with indeterminate entries.  Changing coordinates, the new   matrix of $g\cdot h$ is  
\begin{equation}\label{A-mult}
\tilde{A} = (g^{[p^e]})^{tr} A g. 
\end{equation}
We need to show that there is a choice of $g$ such that the entries of $\tilde{A}$ satisfy the hypothesis of Lemma  \ref{adela}. 

Thinking of the entries of $g$ as indeterminates $Y_{ij}$,  the matrix product  (\ref{A-mult}) has entries
\begin{equation}\label{A-entry}
 \tilde{A}_{ij} =  \sum_{1\leq k,\ell  \leq n} a_{\ell k} Y^{p^e}_{\ell  i}Y_{kj},
 \end{equation} which are homogeneous polynomials in the $Y_{ij}$.
It suffices to prove that there exist values of the $Y_{ij}$ that satisfy the equations
\begin{equation}\label{system}
\tilde{A}_{in} \tilde A_{nn}^{p^e-1} =  \tilde{A}_{ni}^{p^e}
 \,\,\,\,\,\,\,\, \,\,\,\,\,\,\,\, {\text{for all }} \ i = 1, 2, \dots, n-1, 
 \end{equation}
and for which the matrix $g$ has a non-zero determinant. 

Plugging in the  expressions (\ref{A-entry}),  the equations (\ref{system})  become
\begin{equation}\label{system1}
\tilde A_{nn}^{p^e-1}\left(\sum_{1\leq k,\ell  \leq n} a_{\ell k} Y^{p^e}_{\ell  i}Y_{kn}\right) = \left(\sum_{1\leq k,\ell  \leq n} a_{\ell k} Y^{p^e}_{\ell  n}Y_{ki}\right)^{p^e}
 \,\,\,\,\,\,\,\, i = 1, 2, \dots, n-1,
\end{equation} 
each of 
which can be rearranged into a linear equation in $Y_{1i}^{p^e},  Y_{2i}^{p^e}, \dots, Y_{ni}^{p^e}$:
\begin{equation}\label{system2}
F_1 Y_{1i}^{p^e} + F_2 Y_{2i}^{p^e} + \dots + F_{n} Y_{ni}^{p^e}  = 0 \,\,\,\,\,\,\,\, \,\,\,\,\,\,\,\,  i=1, \dots n-1,
\end{equation} 
where the coefficients $F_j$ of $ Y_{ji}^{p^e}$ are 
$$
F_j = \left(\tilde A_{nn}^{p^e-1}\sum_{k=1}^na_{jk} Y_{kn}   -\sum_{\ell=1}^n a_{\ell j}^{p^e}Y_{\ell n}^{p^{2e}}  \right)    \,\,\,\,\,\,\,\, \,\,\,\,\,\,\,\,  j=1, \dots, n.
$$
The key things to notice here are that the coefficient  $F_j$ 
of $Y_{ji}^{p^e}$ in the equations (\ref{system2}) is the {\it same} for each $i = 1, \dots, n$, and that it is a polynomial in  {\it only} $Y_{1n}, \dots, Y_{nn}$, the entries of the last column of the matrix $g$.  Thus the $F_1, \dots, F_{n}$ form a homogenous system of polynomials in  the $n$ indeterminates  $Y_{1n}, \dots, Y_{nn}$ of degree $p^{2e}$.

We claim that there is a choice of values for $Y_{1n}, \dots, Y_{nn}$, not all zero, for which all $F_1, \dots, F_n$ vanish.  In this case, we can take $g$ to be the matrix that has these values as its final column, with any choice of the first $n-1$ columns that makes $g$ invertible. For this choice of $g$, we will have proved that changing coordinates by $g$, the form $h$ can be put into the desired form of Lemma \ref{adela}. Thus the proof is complete once we have found a non-zero solution to the system $\{F_1=F_2=\dots = F_n =0\}$.

To prove this claim, we invoke Hilbert's Nullstellensatz:  provided the ideal generated by $F_1, F_2, \dots,  F_n$  in $k[Y_{1n}, Y_{12}, \dots, Y_{nn}]$ is not  $\langle Y_{1n}, Y_{12}, \dots, Y_{nn}\rangle $-primary,  the Nullstellensatz provides the needed non-zero solution.
 But expanding out the vacuously true expression $\tilde{A}_{nn}\tilde A_{nn}^{p^e-1} = \tilde A_{nn}^{p^e}$ produces the following relation: 
$$
Y_{1n}^{p^e}F_1 + Y_{2n}^{p^e}F_2 + \dots + Y_{nn}^{p^e} F_n = 0.
$$ 
Since $Y_{nn}^{p^e}$  has degree $p^e$, it cannot be in the ideal generated by the elements  $F_1, F_2, \dots,  F_{n-1}$, which have degree $p^{2e}$, 
 showing that  $\{F_1, F_2, \dots,  F_n\}$ is {\it not} a regular sequence. Thus the ideal $\langle F_1, F_2, \dots,  F_n\rangle$ has height strictly less than $n$. Thus 
 the Nullstellensatz gives the  needed non-zero solution to the system $F_1=F_2=\dots = F_n =0$. The theorem is proved.
 \end{proof}

\begin{remark}\label{JK}
J\'anos Koll\'ar suggested a different argument for Theorem \ref{fullrank} based on showing the stabilizer of the $GL_n$-action on the space of Frobenius forms is zero dimensional.
\end{remark}


\section{Isomorphism Types of Extremal Singularities.}\label{IsoTypes}

In this section, we classify Frobenius forms up to linear changes of coordinates, working always over an algebraically closed 
field $k$ of  characteristic $p>0$. Throughout, $q$ denotes a fixed power of $p$.

\begin{theorem}\label{partitions} There is a bijection between partitions of $n$ and  non-degenerate Frobenius forms  in  $n$ variables (with  fixed degree), up to change of coordinates.

The bijection sends a partition  $n=\sum s$ to  the  Frobenius form class represented by a  block diagonal matrix
with exactly one  $s\times s$ block $J_s$ along the diagonal  for each summand $s$ in the partition; the blocks  $J_s$ are
\begin{equation}\label{J}  J_1 = [1] \,\,\,\,\,\,\,\,\, {\text{and}} \,\,\,\,\,\,\,
J_s \,\, =\,\,  \begin{bmatrix} 
0 & \textcolor{blue}{1} & 0   & \dots & 0  \\
0 & 0  & \textcolor{blue}1 &    \dots & 0 \\
\vdots & \vdots & \ddots &\textcolor{blue} \ddots & \vdots \\
0 & 0 &   \dots  & 0 & \textcolor{blue}1  \\
0 & 0 &  \dots & 0 &  0\\
\end{bmatrix}  \,\,\,\,\, {\text{for}} \,\,\, s\geq 2,
\end{equation}
where the matrix $J_s$ for $s\geq 2$ has $1$'s on the super-diagonal and zeros elsewhere. 
\end{theorem}

\begin{example} There are three partitions of three: $3=1+1+1, \,\, 3=1+2$ and $3=3$. So Theorem \ref{partitions} says there are three equivalence classes of  non-degenerate Frobenius forms in three variables, corresponding, respectively, to the three matrices (with the blocks $J_s$ shaded)
$$
\left[
\arraycolsep=3.3pt
\begin{array}{ccc}
\cellcolor{blue!15}1 & 0 & 0 \\ 
0 & \cellcolor{blue!15}1 & 0 \\ 
0 & 0 & \cellcolor{blue!15}1 
\end{array}
\right],
\,\,\,\,\,\,\,\,
\left[
\arraycolsep=3.3pt
\begin{array}{ccc}
\cellcolor{blue!15} 1 & 
0 & 
0 \\ 
0 & 
\cellcolor{blue!15} 0 &
\cellcolor{blue!15} 1 \\ 
0 & 
\cellcolor{blue!15} 0 & 
\cellcolor{blue!15} 0 
\end{array}
\right]
\,\,\,\,\,\,{\text{and}} \,\,\,\,\,\,\,\,
\left[
\arraycolsep=3.3pt
\begin{array}{>{\columncolor{blue!15}} c >{\columncolor{blue!15}} c >{\columncolor{blue!15}} c}
0 & 1 & 0 \\ 
0 & 0 & 1 \\ 
0 & 0 & 0 
\end{array}
\right].
$$

\noindent
 These determine, respectively, the forms $x^{q+1} + y^{q+1} + z^{q+1}, \,\, x^{q+1} + y^qz\, $ and  $\,x^qy+ y^qz$.  (These were classified already in \cite{cubicspaper}.)
\end{example}

\begin{definition}\label{standard}
 A Frobenius form   (or its matrix) constructed from a partition in this way will be said to be in  {\bf standard form}.
 \end{definition}
 
Theorem \ref{partitions} says that every Frobenius form can be brought into one and only one standard form, up to permuting the blocks, by  a linear change of coordinates.
The full rank  $n$ case is the content of Theorem \ref{fullrank} and the rank $n-1$ was proved by Hoang \cite{Hoang}.

 \smallskip
 
\begin{remark}\label{HermFrob}
 The special class of Hermitian forms (see definition in  Remark \ref{Hermitian}) are uniquely determined by their rank: up to changing coordinates, we have only $x_1^{q+1}+ \dots + x_r^{q+1}$
\cite[4.1]{BC}.
 In particular,  there are many more Frobenius forms than Hermitian forms. \end{remark}

 \smallskip
 \begin{remark} \label{HWCompare} The standard forms of Theorem \ref{partitions} look similar to those in \cite[Satz 11]{HesseWitt}, but those are normal forms for a {\it different} action of GL($n$) on  $n\times n$ matrices over a field of characteristic $p$ defined by $A \mapsto gAg^{[-p]}$ (where $g^{[-p]}$ means the inverse of  $g^{[p]}$). Indeed, while the  {\it number} of distinct orbits is the same as for the action  $A \mapsto  (g^{[p]})^{tr}Ag$,  the orbits themselves  are  {\it different.}  For example,  $\begin{bmatrix} 1 & 1 \\ 0 & 0 \end{bmatrix}$  and $\begin{bmatrix} 1 & 0 \\ 0 & 0 \end{bmatrix}$ 
 are equivalent under the action in \cite[Satz 11]{HesseWitt}, but these matrices do not define equivalent Frobenius forms. 
  \end{remark}

 \medskip
\subsection{The proof.}
 Theorem \ref{partitions} will be proved  in two steps. We first show that every Frobenius form can be put into  standard form by a linear change of coordinates (Theorem \ref{stand}).
   We then  show that  if two Frobenius forms $f_{\alpha}$ and $f_{\beta}$ in standard form are equivalent, then their partitions $\alpha$ and $\beta$ are the same (Lemma \ref{final}).

 Towards the first step, it is helpful to consider a broader class of Frobenius forms:

\smallskip
\begin{definition} A  matrix
 is  {\bf sparse} if it has 
at most one non-zero entry in each column   and  in each row,  and these nonzero entries are all $1$. A
 Frobenius form is {\bf sparse} if its matrix is sparse.  \end{definition}
 Put differently, 
 a sparse Frobenius form is one of the form
$$
f = \sum_{k=1}^r x_{i_k}^qx_{j_k}
$$
where the row indices $\{i_1,..., i_r\}$ are all distinct and the column  indices $\{j_1,\dots, j_r\}$ are all distinct.
Sparse forms are always non-degenerate (if we view them as forms in the variables  that explicitly appear in them); this follows from Lemma \ref{Janet} and is proved carefully in   \cite[5.3]{WicaPaper}.

The following crucial  lemma reduces the proof of Theorem \ref{partitions} to  a combinatorial argument:

\begin{lemma}
 \label{Emily} Every  Frobenius form  is projectively equivalent to  a {\it sparse} Frobenius form.
\end{lemma}

\begin{proof}
Consider a non-degenerate Frobenius form of embedding dimension $n$ and rank $r$.  We induce on $n$ to show that after a change of coordinates,  
its matrix can be made sparse.  

When $n\leq 2$, the form  is projectively equivalent to  $x^{q+1}, \,  x^{q+1} + y^{q+1}$\, or \,
 $x^{q}y$.  The corresponding  matrices are
   \[ \begin{bmatrix} 1 \end{bmatrix}, \,\,\,
   \begin{bmatrix}
   1& 0 \\
   0 & 1 \\
   \end{bmatrix} \,\,\,\,\,\,{\text{and}} \,\,\,\,\,
   \begin{bmatrix}
   0 & 1 \\
   0 & 0 \\
   \end{bmatrix},    \] which are sparse.

Now assume $n\geq 3$.  By Theorem \ref{fullrank}, we need only consider the case where  $r<  n$, since the diagonal form $\sum_i x_i^{q+1}$ is sparse. 

  We  first claim that after a change of coordinates,  the matrix $A$ of the Frobenius form can be assumed to be in the  block form 
\begin{equation} \label{form-1}
\bbordermatrix{ 
&  2r-n  & n-r & n-r  \    \cr
    2r-n  \ & B & C & \mathbf{0}   \cr 
    n-r & \mathbf{0} & \mathbf{0} & I   \cr
    n-r    &\mathbf{0} & \bf \mathbf{0} &  \mathbf{0} \cr}
\end{equation}
where $I$ is an $(n-r) \times (n-r)$  identity matrix.
This follows from Lemma  \ref{Janet}: writing the form $x_1^qL_1 + \dots + x_r^qL_r$,  the span of the linear forms
$\{x_1, \dots, x_r, L_1, \dots, L_r\}$ is $n$ dimensional, for otherwise the form is degenerate. So there is some set of $n-r$  forms $L_i$ which, together with 
$\{x_1, \dots, x_r\}$, are linearly independent, and hence can be taken as our coordinates. Permuting  the variables if needed, we can assume these are the last $n-r$ of the $L_i$.
That is,  renaming so that $L_r=x_n$, $L_{r-1} = x_{n-1}, \dots,  L_{2r-n+1} = x_{r+1}$, the form can be assumed
$$
x_1^qL_1' + x_2^qL_2'  + \dots + x_{2r-n}^q L_{2r-n}'+  \underbrace{x_{2r-n+1}^qx_{r+1}  + \dots + x_{r-1}^qx_{n-1}+ x_r^qx_n}_{{\tiny{\text{produces the identity block of size} \,\,n-r}}},$$
 whose matrix nearly has  our claimed form \eqref{form-1}. The only issue is that  the linear forms $L_i'$ could involve the variables $x_{r+1}, \dots, x_n$ which would place non-zero entries above the identity matrix block. However, in this case, we can add multiples of the middle $n-r$ rows upwards to clear out any offending non-zero entries; the corresponding column operations may change the entries of $B$ but not any of the relevant blocks of zero (or the identity block).
 
 The submatrix $B$ in \eqref{form-1} represents a Frobenius form in $2r-n$ variables, which can be written non-degenerately in, say,   $m$ variables. Let $s$ be the rank of $B$. Thus after a change of coordinates, $B$ is  equivalent to a matrix with the block form
 \begin{equation}\label{B-form}
\bbordermatrix{ 
&  m & 2r-n-m      \cr
    m & D & \mathbf{0} \cr
    2r-n-m  & \mathbf{0} & \mathbf{0} \cr}
\end{equation}
 where $D$ is non-degenerate of rank $s$.  By induction, we may assume that $D$ is  sparse.   As a sub-matrix of $A$, the transformations of $B$ leading to this simplified form can be achieved by operations on $A$ that preserve each block of zeros, as well as the identity sub-matrix in \eqref{form-1}, though they may alter $C$.   Relabling the variables involved in $D$,  we can assume its last $m-s$ rows are zero. 
 Hence we can assume that $A$ has form \eqref{form-1}, where  $B$ is sparse, and its last $2r-n-s$ rows are zero. 
 
The columns of $B$ are the  standard basis elements $e_1, \dots, e_s$,  and so by adding multiples of them to   middle block of $n-r$ columns of $A$ in \eqref{form-1},  we can clear out the first $s$ rows of $C$.  The corresponding row operations  add multiples of rows $\{1, \dots, 2r-n\}$  to the middle block of $n-r$ rows of $A$ in \eqref{form-1}, possibly destroying those zero blocks. However,  this can be corrected by adding multiples of the last $n-r$ columns of $A$ to its first $r$ columns, as the corresponding row operations have no further effect. 

We have reduced to assuming that $A$ has the form
\begin{equation}\label{B-form-2}
\bbordermatrix{ 
&  2r-n  & n-r & n-r  \    \cr
    s & E & \mathbf{0} & \mathbf{0} \cr
    2r-n-s  \  & \mathbf{0} & F & \mathbf{0}   \cr 
    n-r  & \mathbf{0} & \mathbf{0} & I   \cr
    n-r    & \bf \mathbf{0} &  \mathbf{0} & \mathbf{0}  \cr}.
\end{equation}
where $E$ is the sparse submatrix of $B$ consisting of its first $s$ rows, and $F$ is the submatrix of $C$ consisting of its last $2r-n-s$ rows.  Next, we claim that if $F$ can be transformed by column operations to a matrix $G$, then $A$ is  equivalent to a matrix of the form \eqref{B-form-2}, with $G$ replacing $F$, and without affecting any other blocks. 

Indeed, column operations on $F$ correspond to column operations on the middle block of $n-r$ columns of $A$.  The corresponding row operations on $A$ leave all zero blocks unchanged but may alter the block $I$ in \eqref{B-form-2}.  However, any such alterations to $I$ can be corrected by Gaussian elimination on the last $n-r$ columns, restoring the matrix $I$ without disturbing anything else, since the corresponding rows are all zero-rows.

To conclude the proof, finally we show that $F$  be transformed into sparse form via column operations. Indeed, using
 Gaussian column elimination, we can put $F$ into a reduced column echelon form, which is sparse because  $F$ is a $(2r-n-s) \times n-r$ matrix of  rank $2r-n-s$. This complete the proof of the lemma.
\end{proof}

\subsection{The directed graph of a sparse form.}
In light of Lemma \ref{Emily}, we have  reduced the proof of Theorem \ref{partitions} to the more combinatorial problem of  classifying equivalence classes of sparse forms. 
For this, it is helpful to make the following definition.

\begin{definition} The {\bf labelled directed graph} of a sparse Frobenius form $f$ is the unique graph   $\Gamma_f$ whose vertices are the  variables $x_1, \dots, x_n$, with an edge from $x_i$ to $x_j$ whenever the term $x_i^qx_j$ appears with non-zero coefficient in $f$.

Put differently, the matrix $A_f$ of a sparse Frobenius form $f$ uniquely determines  a  labelled directed graph $\Gamma_f$ whose  adjacency matrix is $A_f$. 
\end{definition}

\begin{remark}\label{graphprop}
The  connected components of $\Gamma_f$  can be only  loops (singleton vertices with one edge in and out), directed chains, and cycles: the sparseness of $f$ means that 
each vertex has at most one edge entering and at most one edge leaving it. 
\end{remark}

\begin{figure}[H]
\centering
\begin{subfigure}[b]{0.5\textwidth}
\centering
\begin{tikzpicture}[node distance=.5cm ]
\node [draw=none] (1) {$x_1$};
\node [draw=none, right=of 1] (2) {$x_2$};
\node [draw=none, below=of 2] (3) {$x_3$};
\node [draw=none, below=of 1] (4) {$x_4$};
\draw [->]  (1) to [out=90,in=90]  (2);
\draw [->]  (3) to [out=-90,in=-90]  (4);
\draw [->]  (2) to [out=0,in=0]  (3);
\draw [->]  (4) to [out=180,in=180]  (1);
\end{tikzpicture}
\vspace{1cm}
\caption{\footnotesize{The  graph of $x_1^qx_2 + x_2^qx_3 + x_3^qx_4 + x_4^qx_1$}}
\end{subfigure}
\begin{subfigure}[b]{0.4\textwidth}
\centering
\begin{tikzpicture}[node distance=.5cm ]
\node [draw=none] (1) {$x_1$};
\node [draw=none, right=of 1] (2) {$x_2$};
\node [draw=none, below=of 2] (3) {$x_3$};
\node [draw=none, below=of 1] (4) {$x_4$};
\draw [->] (1) to [out=180,in=90,looseness=8] (1);
\draw [->] (2) to [out=0,in=90,looseness=8] (2);
\draw [->] (3) to [out=0,in=-90,looseness=8] (3);
\draw [->] (4) to [out=180,in=270,looseness=8] (4);
\end{tikzpicture}
\vfill
\caption{\footnotesize{The graph of $x_1^{q+1} + x_2^{q+1} + x_3^{q+1}+x_4^{q+1}$}}
\end{subfigure}
\end{figure}

 \begin{remark}\label{graph-prop}
By Lemma \ref{Janet},  the graph of a sparse Frobenius form $h$ satisfies:
\begin{enumerate}
\item[(a)] The number of vertices is the embedding dimension of $h$;
\item[(b)] The number of edges is the rank of $h$.
\end{enumerate}
\end{remark}

The  directed graph of a sparse Frobenius form is {\it not}  invariant under change of coordinates, as the example above shows. 
The next lemma says that to understand projective equivalence,  we can restrict our attention to graphs whose components are either loops or chains:

\begin{lemma}\label{sparse-standard} 
Every sparse Frobenius form is equivalent to a sparse Frobenius form whose directed graph has no cycles.
\end{lemma}

\begin{proof}[Proof of Lemma \ref{sparse-standard}] Let $f$ be  a sparse Frobenius form.
If its  graph $\Gamma_f$ contains an  $\ell$-cycle (say $x_{i_1} \rightarrow x_{i_2} \rightarrow \dots \rightarrow x_{i_\ell} \rightarrow x_{i_1}$), then the form $f$ can be written as 
$$f = x_{i_1}^qx_{i_2} + x_{i_2}^qx_{i_3} + \dots + x_{i_{\ell-1}}^qx_{i_\ell} + x_{i_{\ell}}^qx_{i_1}  + h$$ where $h$ does not involve the variables 
$x_{i_1}, \dots, x_{i_\ell}.$ Now change coordinates involving only the variables $x_{i_1}, \dots, x_{i_\ell}$  to transform $f$ into the equivalent form 
$$\tilde f = x_{i_1}^{q+1}+ x_{i_2}^{q+1} + \dots +  x_{i_{\ell}}^{q+1}  + h$$
(Theorem \ref{fullrank}).
This transforms  $\Gamma_f$ into the graph $\Gamma_{\tilde f}$ in which the $\ell$-cycle has been broken into $\ell$ loops, but whose remaining components are the same as in $\Gamma_f$.
Repeating the process on each cycle in the graph, we eventually arrive at a graph with no cycles.
\end{proof}

\begin{remark}
An alternate way to prove Lemma \ref{sparse-standard} is to observe that we can strengthen the inductive hypothesis in the proof of Lemma \ref{Emily} to show that every Frobenius form is equivalent to a sparse Frobenius form whose matrix is {\it upper triangular}. Then notice that the graph of such a sparse form can not contain any cycles.
\end{remark}
 \medskip

A finite directed graph whose components are all loops or directed chains is essentially a partition---namely the partition of its vertices into components.
We can put all these ideas together  to complete the first step in our proof of Theorem \ref{partitions}:

\begin{theorem}\label{stand}
Every Frobenius form is projectively equivalent to one in standard form.
\end{theorem}

\begin{proof} 
Using Lemma \ref{Emily},  we assume   $f$ is sparse, and using Lemma \ref{sparse-standard}, we assume its graph $\Gamma_f$ has no cycles. Thus  $\Gamma_f$ is a 
 disjoint union of loops and directed  chains (Remark \ref{graphprop}). Relabel the vertices  so
 all arrows point from $x_i$ to $x_{i+1}$ (or back to $x_i$ itself).  Focusing on one component, say of cardinality $s$, note that its adjacency matrix  is the matrix  $J_s$ as defined in (\ref{J}). So the full adjacency matrix $A_f$ of $\Gamma_f$ is a block diagonal matrix with blocks of the type $J_s$ for different values of $s$, one for each component of $\Gamma_f$. In other words, $A_f$ is a matrix in standard form (Definition \ref{standard}). 
 Thus the corresponding Frobenius form $f$ is in standard form  as well.
\end{proof}

\medskip 
To complete the classification, we need to show that the Frobenius forms associated to {\it different} partitions are not equivalent.
Therefore, it suffices to establish the following lemma:

\begin{lemma}\label{final} Let $f$ and $g$ be Frobenius forms  in standard form. If $f$ and $g$ are  equivalent, then the components of their associated 
 graphs $\Gamma_f$ and $\Gamma_g$  determine the same partition.
\end{lemma}

\begin{proof} 
We  induce on the embedding dimension $n$. The cases $n\leq 2$ are easy and were listed in the proof of Lemma \ref{Emily}. In addition, all full rank Frobenius forms are equivalent to   $\sum_{i=1}^n {x_i}^{q+1}$ (Theorem \ref{fullrank}), so correspond to the partition of all $1$'s. So we assume $n\geq 3$ and that  the rank $r$ satisfies $r< n$.

The graphs $\Gamma_f$ and $\Gamma_g$  partition  the vertices into components, determining partitions $\gamma_f$ and $\gamma_g$, respectively, of $n$.
The partition $\gamma_f$ can be written
$$ n = \underbrace{1 + \dots + 1}_{d_1} + \underbrace{2 + \dots + 2}_{d_2} + \underbrace{3 + \dots + 3}_{d_3} + \dots +  \underbrace{t + \dots + t}_{d_t},$$
where $d_i$ is the number of times the integer $i$ appears in the partition (note that some $d_i$ can be zero). Likewise, the partition $\gamma_g$ can be written 
$$ n = \underbrace{1 + \dots + 1}_{e_1} + \underbrace{2 + \dots + 2}_{e_2} + \underbrace{3 + \dots + 3}_{e_3} + \dots +  \underbrace{t + \dots + t}_{e_t}.$$

\begin{formula} \label{formula} Using Remark \ref{graph-prop}, we  count vertices and edges to get the following formulas for the embedding dimension and rank of $f$ and $g$ in terms of the partitions:
\begin{enumerate}
\item[(i)] $n =  \sum_{i\geq 1} i d_i = \sum_{i\geq 1} i e_i $.
\item[(ii)] $r =    d_1 + \sum_{i\geq 2} (i-1) d_i =   d_1 + \sum_{i\geq2}(i-1) e_i $.
\item[(iii)] $n-r = d_2 + \dots + d_t = e_2 + \dots + e_t.$
\end{enumerate}
\end{formula}

Say that a vertex of a directed graph is {\bf terminal} if there is no edge emanating from  it.   The number of terminal vertices in the graph of a sparse form 
  is $n-r$, by  Formula \ref{formula}.
 Let $\{y_1, \dots, y_{n-r}\}$ denote  the terminal  variables of  $f$; relabeling the variables of  $g$, we can assume $\{y_1, \dots, y_{n-r}\}$ are also the terminal variables of $g$.
Denote the remaining variables--- the non-terminal variables---  by $\{x_1, \dots, x_r\}$.  These are the ones appearing in $f$ and $g$ with exponent $q$.

Now define the {\bf pre-terminal} variables of $f$   to be those connected to a terminal variable  of  $\Gamma_f$ by an edge, and note that there are exactly  $n-r$ pre-terminal variables.  Call them
$\{x_{i_1}, \dots, x_{i_{n-r}}\}$. Relabeling the variables of $g$, we can assume that these are also the pre-terminal variables of $g$.  
So we can write 
 \begin{equation}\label{eq4}
 \begin{aligned}
f &= f_1(x_1, \dots , x_r) + x_{i_1}^qy_1 + x_{i_2}^qy_2 + \dots + x_{i_{n-r}}^qy_{n-r}\\
g &= g_1(x_1, \dots , x_r) + x_{i_1}^qy_1 + x_{i_2}^qy_2 + \dots + x_{i_{n-r}}^qy_{n-r}
\end{aligned}
\end{equation}
where $x_{i_\ell}$ is the pre-terminal variable corresponding to  terminal variable $y_\ell$.

\begin{claim} \label{claim} With notation as above, 
any  linear change of coordinates  $\phi$ such that  $\phi(f) = g$ must  preserve the ideal $\langle x_{i_1}, \dots, x_{i_{n-r}}\rangle$ generated by the pre-terminal variables. 
\end{claim}

To justify  Claim \ref{claim}, 
note that the defining ideal of singular locus of  both $f$ and $g$ is generated by $\{x_1, \dots, x_r\}$ (Proposition \ref{sing}).  In particular,  $\phi$ must preserve the ideal  $\langle x_1, \dots, x_r\rangle$. So if $L_i$ denotes the image $\phi(x_i)$, then $L_i$ must be a linear form in  $\{x_1, \dots, x_r\}$.

Let $M_1, \dots, M_{n-r}$ be the images of the terminal variables $y_1, \dots, y_{n-r}$ under $\phi$.  Applying $\phi$ to the first line of (\ref{eq4}), we have 
 \begin{equation}\label{eq1}
 \begin{aligned}
\phi(f)  &= f_1(L_1, \dots , L_r) + L_{i_1}^qM_1 + L_{i_2}^qM_2 + \dots + L_{i_{n-r}}^qM_{n-r}. \\
 &= g_1(x_1, \dots , x_r) + x_{i_1}^qy_1 + x_{i_2}^qy_2 + \dots + x_{i_{n-r}}^qy_{n-r}
\end{aligned}
\end{equation}
Writing  $$M_i = c_{i1}x_1 + \dots + c_{ir}x_r + b_{i1}y_1 + \dots + b_{in-r}y_{n-r},$$  and comparing the coefficient of $y_m$ in the  equal expressions (\ref{eq1}), we see that
$$
  x_{i_m}^q \,= \,\sum_{\ell = 1}^{n-r} b_{\ell m} L_{i_\ell}^q = \left( \sum_{\ell = 1}^{n-r} b_{\ell m}^{1/q} L_{i_\ell} \right)^q.
$$
 In particular, 
$$
\left<x_{i_1}, \dots, x_{i_{n-r}}\right> \subset \left<L_{i_1}, \dots, L_{i_{n-r}} \right>  = \left<\phi(x_{i_1}), \dots, \phi(x_{i_{n-r}}) \right>.
$$
Since both ideals are generated by $n-r$ linearly independent linear forms, they must be the same.
This says that  $\phi$ preserves the ideal of pre-terminal variables. The claim is proved. 

Given Claim \ref{claim}, it follows that 
 $\phi$ induces an isomorphism 
\begin{equation}\label{eq3}
\bar{\phi}: \,\,\frac{k[x_1, \dots, x_n]}{\left<x_{i_1}, \dots, x_{i_{n-r}}\right>}\,\,  \longrightarrow  \,\,  \frac{k[x_1, \dots, x_n]}{\left<x_{i_1}, \dots, x_{i_{n-r}}\right>},
\end{equation}
which we view as a change of coordinates for  the quotient polynomial ring.
Letting $\bar{f}$ and   $\bar{g}$ denote the images of $f$ and $g$, respectively,  in the quotient ring, we see that $\bar{\phi}$ defines a  equivalence 
between the Frobenius forms  $\bar{f}$ and   $\bar{g}$.

Now let us examine the graph $\Gamma_{\bar{f}}$ \, of ${\bar{f}}$. Killing the pre-terminal variables zeros out the monomials of $f$ in which the pre-terminal variables appear. 
This removes the  pre-terminal vertices from each directed chain of $\Gamma_f$, as well as any edge connected to them and any vertices left isolated by the process.
So the graph \,\,$\Gamma_{\bar{f}}$ \,\, of \, $\bar{f}$ is  obtained by removing  chains of length $2$ and $3$ from $\Gamma_f$, and  turning  chains of length $\ell \geq 4$ into  chains of length $\ell - 2$.

For example, the figure below shows a directed graph $\Gamma_f$ of a Frobenius form $f$, highlighting the  pre-terminal vertices and every edge connected to a pre-terminal vertex:
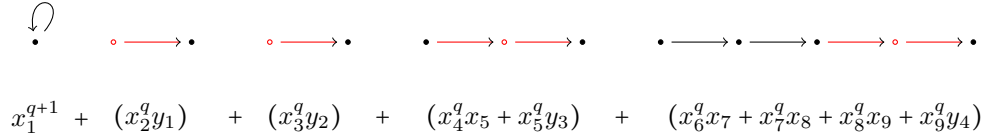
\begin{center}
\begin{tikzpicture}[node distance=.75cm ]
\node [draw=none] (1) {};
\node [draw=none, right=of 1] (2) {};
\node [draw=none, right=of 2] (3) {};
\node [draw=none, right=of 3] (4) {};
\node [draw=none, right=of 4] (5) {};
\node [draw=none, right=of 5] (6) {};
\node [draw=none, right=of 6] (7) {};
\node [draw=none, right=of 7] (8) {};
\node [draw=none, right=of 8] (9) {};
\node [draw=none, right=of 9] (10) {};
\node [draw=none, right=of 10] (11) {};
\node [draw=none, right=of 11] (12) {};
\node [draw=none, right=of 12] (13) {};
\node[draw=none] at (.2,-1) (x1){\footnotesize{$x_1^{q+1} \;\;+$}};
\node[draw=none] at (1.9,-1) (y1){\footnotesize{$(x_2^{q}y_1) \; \; \;\;\;\; +$}};
\node[draw=none] at (3.9,-1) (y2){\footnotesize{$(x_ {3}^qy_2) \;\;\;\;\; +$}};
\node[draw=none] at (6.55,-1)(y3){\footnotesize{$(x_4^qx_5 + x_5^qy_3) \;\;\;\;\; +$}};
\node[draw=none] at (10.5,-1) (y4){\footnotesize{$(x_6^qx_7 + x_7^qx_8+ x_8^qx_9 + x_9^qy_4)$}};
\draw[fill] (1) circle [radius=0.03];
\draw[draw=red] (2) circle [radius=0.03];
\draw[fill] (3) circle [radius=0.03];
\draw[draw=red] (4) circle [radius=0.03];
\draw[fill] (5) circle [radius=0.03];
\draw[fill] (6) circle [radius=0.03];
\draw[draw=red](7) circle [radius=0.03];
\draw[fill] (8) circle [radius=0.03];
\draw[fill] (9) circle [radius=0.03];
\draw[fill] (10) circle [radius=0.03];
\draw[fill] (11) circle [radius=0.03];
\draw[draw=red] (12) circle [radius=0.03];
\draw[fill] (13) circle [radius=0.03];
\draw [draw=red,->] (2) -- (3);
\draw [draw=red,->]  (4) -- (5);
\draw[draw=red,->]  (6) -- (7);
\draw[draw=red,->] (7) -- (8);
\draw [->] (9) -- (10);
\draw [->] (10) -- (11);
\draw [draw=red,->]  (11) -- (12);
\draw [draw=red,->] (12) -- (13);
\draw [->] (1) to [out=50,in=100,looseness=10] (1);
\end{tikzpicture}

\end{center}
Modding out the pre-terminal variables $\{x_2, x_3, x_5, x_9\}$  produces the Frobenius form $\bar{f}$ with graph $\Gamma_{\bar{f}}$:
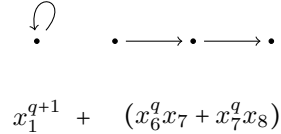
\begin{center}
\begin{tikzpicture}[node distance=.75cm ]
\node [draw=none] (1) {};
\node [draw=none, right=of 1] (2) {};
\node [draw=none, right=of 2] (3) {};
\node [draw=none, right=of 3] (4) {};
\node[draw=none] at (.2,-1) (x1){\footnotesize{$x_1^{q+1} \;\;+$}};
\node[draw=none] at (2.2,-1) (y4){\footnotesize{$(x_6^qx_7 + x_7^qx_8)$}};
\draw[fill] (1) circle [radius=0.03];
\draw[fill] (2) circle [radius=0.03];
\draw[fill] (3) circle [radius=0.03];
\draw[fill] (4) circle [radius=0.03];
\draw [->] (2) -- (3);
\draw [->] (3) -- (4);
\draw [->] (1) to [out=50,in=100,looseness=10] (1);
\end{tikzpicture}

\end{center}

In particular, the partition $\gamma_{\bar{f}}$ is
$$
r =  \underbrace{1 + \dots + 1}_{d_1} + \underbrace{2 + \dots + 2}_{d_4} + \underbrace{3 + \dots + 3}_{d_5} + \dots +  \underbrace{(t-2) + \dots + (t-2)}_{d_t},
$$
and the partition $\gamma_{\bar{g}}$ is
$$
r =  \underbrace{1 + \dots + 1}_{e_1} + \underbrace{2 + \dots + 2}_{e_4} + \underbrace{3 + \dots + 3}_{e_5} + \dots +  \underbrace{(t-2) + \dots + (t-2)}_{e_t}.
$$
By induction, because $\bar{f}$ and $\bar{g}$ are equivalent and of embedding dimension less than $n$, their partitions are the same, so that  $d_i=e_i$ for all $i\neq 2, 3$.

Since $\gamma_f$ and $\gamma_g$ are both partitions of $n$,   Formula \ref{formula}(i) now implies that  $2d_2+3d_3= 2e_2+3e_3$. But Formula \ref{formula}(iii) also gives 
$d_2+d_3 = e_2+e_3.$
Together, these two equations imply finally that $d_2=e_2$ and $d_3=e_3$.
 Thus $f$ and $g$ had the same partition to start.
\end{proof}




\section{Geometric Properties of Extremal Singularities}\label{Geometry}


 Smooth projective  varieties defined by (certain special)  Frobenius forms have long been understood to be extremal in various ways, going at least back to Beniamino Segre \cite{segre}. 
 It is easy to see that they contain many linear subspaces, for example,  which can be used to show that they are extremal from the point of view of 
 containing rational points; see \cite{kollar14}, \cite{BC} and \cite{HommaKim}.

In this section, we collect a few interesting properties of extremal  hypersurfaces. By {\bf extremal hypersurface}, we mean a  {\it projective}  hypersurface 
 defined by a (not necessarily reduced) Frobenius form; in the reduced case, an extremal hypersurface is a projective hypersurface for which the affine cone over it is an extremal singularity.

\subsection{Hyperplane Sections} It is easy to see that every  hyperplane section of {an}  extremal hypersurface  is extremal. Somewhat surprisingly, the converse is also true:

\begin{theorem}\label{sections}
If $X$ is an extremal hypersurface, then so is every hyperplane section (which is not just a component of $X$). Conversely,
if the ground field is algebraically closed and $n\geq 3$, then any hypersurface $X \subset \mathbb P^n$ with the property that 
 all its hyperplane sections are extremal must itself be  extremal. 
\end{theorem}

\begin{example}
The second claim of Theorem \ref{sections} is false for $n<3$. For example, the plane curve of characteristic two defined by the vanishing of $x^3+y^3+z^3+xyz$ is not extremal, but every hyperplane section is extremal. Indeed, every polynomial of degree three in two variables $x, y$ is contained in $\langle x^2, y^2\rangle$.\end{example}

\begin{proof}
Since both degree and inclusion in  $\mathfrak m^{[p^e]}$  are preserved under taking the quotient by a linear form, 
the first statement is clear. 

For the converse, we set up some notation.  For a form  $f$  and a  linear form $L$, let $\bar f$ denote the form  $f\bmod L$ in the polynomial ring $k[x_0, x_1, \dots, x_n]/\langle L\rangle.$

Now suppose that  $f$ is  the defining equation of  the hypersurface $X$ with the property that every hyperplane section is  extremal.  Then  $\bar f$  is a Frobenius form (for all choices of $L$) so $f$ has degree $p^e+1$ for some $e$.

Write $f$ {\it uniquely} as $f = \sum_{i=0}^nx_i^{p^e} L_i  + g$, where the $L_i$ are linear forms, and $g$ is some form {\it none of whose monomials are divisible by any $x_i^{p^e}$}.  We need to show that $g$ is zero. For this, it suffices to show that $g$ is  divisible by infinitely many distinct (up to scalar multiple) linear forms. 
 
Fix any linear form $L$. By hypothesis, $\bar f  $ is a Frobenius form.  Since the set of  Frobenius forms is closed under addition,  also 
$\bar g$ is a Frobenius form. 
 Now if $L=x_i$,  the restriction on the monomials of $g$ implies  that $\bar g=0$---that is, that $x_i$ divides $g$ for each $i$. So  without loss of generality
$$
g =  (x_0x_1\dots x_n) h,
$$
where $h$ is a form of degree $p^e+1 - (n+1)$.

Next, we consider what happens when  $L= x_0 -  c x_1$ for some $c\in k$.
Using the isomorphism
$$
k[x_0, x_1, \dots, x_n]/\langle x_0-cx_1\rangle \longrightarrow  k[y_1, \dots, y_n] \,\,\,\,\,\,\,\,\,\, \,\,\,\,\,\,\,\, \begin{cases}
x_0 \mapsto cy_1  \\ x_i\mapsto y_i \,\,\,\,\,\,\,\,\  i\geq 1\end{cases}
$$
we see that because $g \bmod L$ is a Frobenius form, also 
$$
y^2_1y_2\dots y_n  \tilde {h}  \in \langle y_1^{p^e}, \dots, y_n^{p^e}\rangle
$$
where $\tilde h$ denotes the image of $h$ in the polynomial ring $k[y_1, \dots, y_n]$.
Because  $y_1, \dots, y_n$ form a regular sequence, this yields (see Lemma \ref{reg}) 
$$
\tilde{h}  \in \langle y_1^{p^e-2},  y_2^{p^e-1}, \dots, y_n^{p^e-1}\rangle.
$$
But the degree of $\tilde h$ is $p^e - n $ which is strictly less than $p^e-2$. So $\tilde h = 0$. In other words, $x_0-cx_1$ divides $h$.   Since $c$ was an arbitrary element of $k$,  $h$ must have at least  $|k|$   distinct linear factors.  Since $k$ is infinite, the proof is complete.
\end{proof}

\begin{corollary} \label{cor1}
If $X$ is a smooth  extremal hypersurface over an algebraically closed field, then all  smooth hyperplane sections are isomorphic.
\end{corollary}

\begin{proof} The hyperplane sections of $X$ are  extremal  by Theorem  \ref{sections}. So the smooth hyperplane sections are cones over  full rank extremal singularities, and hence   all projectively equivalent to the diagonal hypersurface  $\sum_{i=1}^n  x_i^{p^e+1}$ by Theorem \ref{fullrank}.
\end{proof}

\begin{remark}\label{rem8}
The converse of Corollary \ref{cor1} is  a theorem of Beauville   \cite{beauville};  restated  in our language, it says that 
if a smooth projective hypersurface $X$  has  the property that all its smooth hyperplane sections are isomorphic to each other,   then $X$ is an extremal  hypersurface.  
\end{remark}

\subsection{Gauss Map} Fix an algebraically closed field. 
Consider a  reduced  closed  subscheme $X\subset \mathbb P^n$ of equi-dimension $d$. 
 The {\bf Gauss map} of $X$ is the rational map 
 $$
 X \dashrightarrow G(d,\,  \mathbb P^n) \,\,\,\,\,\,\,\,\,  \,\,\,\,\,\,\,\,\,  x\mapsto T_xX
 $$
 sending each smooth point $x$ to its embedded projective tangent space $T_xX$, considered as a point in the Grassmannian of $d$-dimensional linear subspaces of $\mathbb P^n$. 
For a hypersurface $X= \mathbb V(f) \subset \mathbb P^n$ defined by a reduced form $f$, the Gauss map can be described explicitly as
 $$
 X \dashrightarrow (\mathbb P^n)^* \,\,\,\,\,\,\,\,\,  \,\,\,\,\,\,\,\,\,  x\mapsto \left[\frac{\partial f}{\partial x_0}: \frac{\partial f}{\partial x_1}: \,\, \cdots \,\,: \frac{\partial f}{\partial x_n} \right].
 $$
This is undetermined along the singular locus of $X$.

It is not hard to see that the Gauss map is finite when $X$ is smooth (without  linear components, which would contract to points under the Gauss map). More generally, the (closure of  the)  image of the Gauss map has  dimension
  $\dim X-\dim {\rm{Sing }}(X) - 1$ \cite[2.8]{Zak}. 

In characteristic zero, the Gauss map of a smooth projective variety is birational, but this can fail in characteristic $p$. Many authors have studied the question of precisely  {\it how} this failure happens, eventually realizing that (at least for hypersurfaces), the issue appears to be  only the {\it inseparability} of the Gauss map; see \cite{wallace},  \cite{Kleiman-Piene}, or \cite{Kaji} for example. 

 A smooth extremal  hypersurface has the property that its Gauss map is highly inseparable---purely inseparable of maximal degree---and its dual hypersurface  is also extremal. The following straightforward statement may be folklore among experts, but we have not found it simply stated in the literature:

\begin{proposition}\label{gauss}
A smooth extremal hypersurface of degree $q+1$ and dimension $d$ has a purely inseparable Gauss map of degree {$q^d$}.  The  dual hypersurface (that is, the image under the Gauss map) is also  a smooth  extremal hypersurface of the same degree.
\end{proposition}

By {\it purely inseparable}, here, we mean that the induced map on generic stalks is a purely inseparable field extension.
\begin{proof}
Suppose $X=\mathbb V(h) $ is an extremal hypersurface in $\mathbb P^n$. Write  $$h =  x_0^{p^e}L_0 + x_1^{p^e} L_{1}  + \dots + x_n^{p^e}L_n= (\vec{x}^{[p^e]})^{tr} A \,  \vec{x}.$$  The Gauss map is
$$
\begin{aligned}
x \,\, \mapsto  \,\, &\,\,  \left[\frac{\partial h}{\partial x_0}: \frac{\partial h}{\partial x_1}: \,\, \cdots \,\,: \frac{\partial h}{\partial x_n}\right] \\
&  = 
\left [\sum_{i=0}^n  a_{i0}x_i^{p^e} \,\, : \, \, \sum_{i=0}^n  a_{i1}x_i^{p^e}\,\, : \,\, \cdots \,\, : \, \, \sum_{i=0}^n  a_{in}x_i^{p^e}\, \right ] \\
& =  [x_0^{p^e}: x_1^{p^e}: \dots: x_n^{p^e}]  A,
\end{aligned}
$$
where $A$ is the (invertible) matrix representing the Frobenius form $h$.
So  the Gauss map  factors as 
$$
[x_0: x_1: \dots: x_n] \mapsto  [x_0^{p^e}: x_1^{p^e}: \dots: x_n^{p^e}]  \mapsto  [x_0^{p^e}: x_1^{p^e}: \dots: x_n^{p^e}]  A.
 $$
Since $A$ is just a linear change of coordinates, we can analyze the induced map on the  generic stalk for the map $
[x_0: x_1: \dots: x_n] \mapsto  [x_0^{p^e}: x_1^{p^e}: \dots: x_n^{p^e}] $ only. Without loss of generality, 
the generic stalk is 
the fraction field of  $k\left[\frac{x_1}{x_0}, \dots, \frac{x_{n}}{x_0}\right]/\left \langle \frac{h}{x_0^{p^e+1}}\right \rangle$, which is a purely transcendental extension of $k$ of transcendence degree $n-1$ generated by 
 the rational functions  $\frac{x_1}{x_0}, \dots, \frac{x_{n-1}}{x_0}$. So the Gauss map on stalks can be viewed as simply the inclusion $k\left((\frac{x_1}{x_0})^{p^e}, \dots, (\frac{x_{n-1}}{x_0})^{p^e}\right)\subset k(\frac{x_1}{x_0}, \dots, \frac{x_{n-1}}{x_0})$, which is purely inseparable of degree $(p^e)^{n-1}$ where $n-1$ is the dimension of the hypersurface.
 
To see that the image is extremal, note because the matrix $A$ is invertible, it suffices to show the $p^e$-th power map on the homogeneous coordinates has extremal image.
But  the relation $ h =( \vec{x}^{[p^e]})^{tr} A \,  \vec{x}$ on the homogeneous coordinates of $X$ implies the relation  $((\vec{x^{[p^e]}})^{[p^e]})^{tr} A^{[p^e]} \,  \vec{x}^{[p^e]}$ on the coordinates $ [x_0^{p^e}: x_1^{p^e}: \dots: x_n^{p^e}]$ of the image. So the image is isomorphic to  the extremal singularity defined by the Frobenius form represented by  $A^{[p^e]}$.
\end{proof}

\begin{remark}\label{just}
Conjecture 2 in \cite{Kleiman-Piene} can be interpreted as predicting  that any smooth hypersurface of degree $d\geq 3$ with the property that its dual hypersurface is smooth {\it must be} defined by a Frobenius form.
(This is known for curves \cite[6.1, 6.7]{Homma89}, \cite[7.8]{Hefez} and surfaces \cite[14]{Kleiman-Piene}.) Thus, in light of  Proposition \ref{gauss}, we should expect  a smooth hypersurface  is extremal if and only if it dual hypersurface is smooth.
\end{remark}

\begin{remark}
If the extremal hypersurface $X$ is not smooth, the proof of Proposition \ref{gauss} shows that its  Gauss map  is  the $p^e$-th power map followed by a linear projection. 
\end{remark}

\subsection{Lines  on  extremal hypersurfaces}

 Extremal  hypersurfaces are extremal also in the behavior of the linear subspaces they contain; see, for example, the discussion in  \cite[\S35]{kollar14}. One simple way to describe this is by looking at the special configurations of intersecting lines on them.

\begin{definition} \label{star}
A configuration of lines in the projective  plane is {\bf perfect star} of degree $d\geq  3$ if it projectively equivalent to  $d$ reduced concurrent lines with slopes ranging through the $d$-th roots of unity. Equivalently, a perfect star of degree $d$ is defined by an equation $x^d-y^d$, where the characteristic of the ground field does not divide $d$.
\end{definition}

\begin{remark}
We could include $d=1, 2$ in Definition \ref{star}, but then every  configuration of lines $d$ forms a perfect star. When $d=3$, a configuration of lines is a perfect star if and only if the three lines are concurrent. The condition becomes more restrictive as $d$ gets larger.
\end{remark}

Perfect stars are clearly  very special configurations of lines---we don't expect most hypersurfaces to contain {\it any}, unless the hypersurface contains an entire plane. 
So the following result emphasizes that  extremal hypersurfaces really have extremal behavior in terms of the configuration of lines they contain:

\begin{proposition}\label{starconfig}  Let $X\subset \mathbb P^n$ be an extremal  hypersurface of degree $q+1,$ where $q$ is a power of the characteristic $p>0$. Suppose $\ell_1$ and $\ell_2$ are intersecting lines  contained in $X$, and let $\Lambda$ be the plane they span.  If $\Lambda$ is not contained in $X$, then the plane section $\Lambda \cap X$  is either a perfect star of degree $q+1$ or the  union of a $q$-fold line and a reduced line.   \end{proposition}

\begin{proof}
If $\Lambda$ does not lie on $X$, then $\Lambda \cap X \subset \Lambda \cong \mathbb P^2$ is {an} extremal curve by Theorem \ref{sections}. Choose coordinates  $\{x, y, z\}$ for $\mathbb P^2$ so that $\ell_1$ and $\ell_2$ are given by the vanishing of $x$ and $y$ in $\Lambda \cong \mathbb P^2$, and let   $\bar{h}$ be the equation of the plane section 
$\Lambda \cap X$. Since this curve is extremal, 
$$
\bar {h}  = x^qL_1 + y^qL_2 + z^qL_3,
$$
for some linear forms $L_i$.
Because $\ell_1$ and $\ell_2$ lie on this curve, we know both $x$ and $y$ divide $\bar h$. This forces  $y\mid L_1, x\mid L_2, $ and $xy\mid L_3$. In particular, $L_3=0$, since its degree is one. So 
$$ 
\bar {h}  = ax^qy + by^qx = xy(ax^{q-1} + by^{q-1})
$$
for some scalars $a, b$. 
So $\bar h$  factors into $q+1$ linear forms, all distinct unless one of $a$ or $b$ is zero. In the former case, we can scale $x$ and $y$ to assume $a=1$ and $b=-1$ to get a perfect star and in the latter case, we have the non-reduced line configuration defined by $x^qy$.    
\end{proof}

{\small
\bibliographystyle{amsalpha}
\bibliography{bibdatabase}
}

\end{document}